\documentclass[11pt,reqno,palentino]{amsart} 

\usepackage{mypack}

\usepackage{latexsym}
\usepackage{a4wide}
\usepackage{amscd}
\usepackage{graphics}
\usepackage{amsmath}
\usepackage{amssymb}
\usepackage{mathrsfs}
\usepackage{accents}
\usepackage{enumerate}

\usepackage[backgroundcolor=white,linecolor=gray]{todonotes}

\DeclareMathOperator{\coker}{coker} 
\DeclareMathOperator{\codim}{codim}

\newtheorem{theorem*}{Theorem}

\author{Alberto Abbondandolo and Thomas Rot}
\begin{document}


\title{The homotopy classification of proper Fredholm maps of index one}
\begin{abstract}
In a previous paper, we classified the homotopy classes of proper Fredholm maps from an infinite dimensional Hilbert manifold to its model space in terms of a suitable version of framed cobordism. We explicitly computed these homotopy classes for non-positive index. In this paper, we compute the homotopy classes of proper Fredholm maps of index one from a simply connected Hilbert manifold to its model space. This classification uses a new numerical invariant for proper Fredholm maps of index one.
\end{abstract}

\maketitle 
\vspace{-.3cm}
\noindent\makebox[\textwidth][c]{
\begin{minipage}[t]{.4\textwidth}
\it \Small ${\ }^{1}$ Ruhr-Universit\"at Bochum\\
\phantom{${\ }^{1}$} Fakult\"at f\"ur Mathematik\\
\phantom{${\ }^{1}$} Universit\"atsstrasse 150\\
\phantom{${\ }^{1}$} Geb\"aude IB 3/65\\
\phantom{${\ }^{1}$} D-44801 Bochum, Germany\\
\phantom{${\ }^{1}$} Email: alberto.abbondandolo@rub.de
\end{minipage}%
\hfill
\begin{minipage}[t]{.4\textwidth}
\it 
\Small ${\ }^{2}$ Vrije Universiteit Amsterdam\\
\phantom{${\ }^{2}$} Departement Wiskunde\\
\phantom{${\ }^{2}$} De Boelelaan 1081a\\
\phantom{${\ }^{2}$} 1081 HV Amsterdam, the Netherlands \\
\phantom{${\ }^{2}$} Email: t.o.rot@vu.nl
\end{minipage}}
\vspace{.3cm}

\centerline{\em In memory of Andrzej Granas}


\section*{Introduction}

Consider a real separable and infinite dimensional Hilbert space $\mH$. The space of Fredholm operators on $\mH$ of index $n$ is denoted by $\Phi_n(\mH)$. 

By Hilbert manifold we mean here a connected paracompact smooth manifold modeled on $\mH$. A smooth map $f: M \rightarrow N$ between Hilbert manifolds is called Fredholm of index $n$ if its differential at every point is a Fredholm operator of index $n$. All the Fredholm maps we consider in this paper are tacitly assumed to be smooth. By Kuiper's theorem \cite{kui65}, the general linear group of $\mH$ is contractible and hence the tangent bundle of any Hilbert manifold is trivial.  By fixing trivializations of the tangent bundles of $M$ and $N$, the differential of a Fredholm map $f:M\rightarrow N$ of index $n$ can be seen as a map
\[
df: M \rightarrow \Phi_n(\mH).
\] 
A Fredholm homotopy between Fredholm maps $f,g: M \rightarrow N$ of index $n$ is a smooth homotopy between $f$ and $g$ that is Fredholm, necessarily of index $n+1$. In this case, we say that $f$ and $g$ are Fredholm homotopic. If the Fredholm maps $f,g: M \rightarrow N$ of index $n$ are Fredholm homotopic, then $f$ and $g$ are homotopic as continuous maps and their differentials
\[
df, \, dg : M \rightarrow \Phi_n(\mH)
\] 
are also homotopic. The converse is also true: Any two Fredholm maps of index $n$ are Fredholm homotopic if and only if they are homotopic as continuous maps and their differentials are homotopic as maps from $M$ to $\Phi_n(\mH)$, see \cite[Proposition 2.24]{et70} and \cite[Theorem 1]{ar20}. Note that $\Phi_n(\mH)$, unlike the space of linear mappings between finite dimensional vector spaces, has a non-trivial topology, and its homotopy groups are given by Bott's periodicity theorem, see Section \ref{topsec} below. 

The question of the homotopy classification of Fredholm maps becomes more interesting, and definitely non-trivial, if one adds the requirement that maps and homotopies should be proper, meaning that the inverse image of any compact set is compact. This is the question we are addressing here, in the special case in which the target space $N$ is the model Hilbert space $\mH$. We denote by
\[
\mathcal{F}_n^{\mathrm{prop}} [M,\mH]
\]
the space of equivalence classes of proper Fredholm maps $f: M \rightarrow \mH$ of index $n$ modulo proper Fredholm homotopies. Building on classical results of Elworthy and Tromba from \cite{et70}, in \cite[Theorem 2]{ar20} we constructed a bijection between $\mathcal{F}_n^{\mathrm{prop}} [M,\mH]$ and a suitable version of framed cobordism. See also \cite{sch64b, geb69, ber77} for related results on the homotopy classification of proper Fredholm maps between Banach spaces whose differential takes values into some contractible subspace of the space of Fredholm operators, such as the space of compact perturbations of the identity.  We refer to Section \ref{prelsec} below for the relevant definitions and for the statements of the results from \cite{ar20} that are needed here. In the case of negative index $n<0$, the bijection mentioned above shows that $\mathcal{F}_n^{\mathrm{prop}} [M,\mH]$ is in one-to-one correspondence with the space of homotopy classes $[M,\Phi_n(\mH)]$, as in the non-proper case, see  \cite[Theorem 3]{ar20}.  For $n=0$, $\mathcal{F}_0^{\mathrm{prop}} [M,\mH]$ is completely determined in terms of $[M,\Phi_n(\mH)]$ and of a suitable degree, see \cite[Theorem 4]{ar20}.

The aim of this paper is to deal with the case $n=1$ and determine $\mathcal{F}_1^{\mathrm{prop}} [M,\mH]$ under the assumption that $M$ is simply connected. In the remaining part of this introduction, we assume the Hilbert manifold $M$ to be simply connected. Recall that $\pi_2(\Phi_1(\mH))=\mZ_2$. We shall say that a  map $A: M \rightarrow \Phi_1(\mH)$ is spin if the homomorphism
\[
\pi_2(A) : \pi_2(M) \rightarrow \pi_2(\Phi_1(\mH)) = \mZ_2
\]
is trivial. We refer to Remark \ref{remspin} below for comments about our choice of using the term ``spin'' in this context. This notion is homotopy invariant, and hence the set of homotopy classes of maps from $M$ to $\Phi_1(\mH)$ decomposes as
\[
[M,\Phi_1(\mH)] = [M,\Phi_1(\mH)]_{\mathrm{sp}} \sqcup [M,\Phi_1(\mH)]_{\mathrm{ns}},
\]
where the first set denotes the set of spin homotopy classes and the second one the set of non-spin ones. The Fredholm map $f: M \rightarrow \mH$ of index one is said to be spin if $df$ is spin. 

The space of spin proper Fredholm maps $f: M \rightarrow \mH$ of index one possesses a $\mZ_2$-valued invariant $\tau$. The definition of $\tau$ builds on the fact that
\[
\pi_2(\Phi_1(\mH),\Phi_1^0(\mH)) = \mZ_2,
\]
where $\Phi_1^0(\mH)$ denotes the subset of $\Phi_1(\mH)$ consisting of surjective operators. If $y\in \mH$ is a regular value of $f$, then $f^{-1}(y)$ is a compact one-dimensional submanifold of $M$, i.e.\ a finite set of embedded circles. If $S$ is one of these circles, by the fact that $M$ is simply connected we can consider a disk $D\subset M$ with boundary $S$. Since $y$ is a regular value, $df$ maps $S$ into $\Phi_1^0(\mH)$ and hence the restriction of $df$ to the disk $D$ defines an element
\[
\sigma(S,df) := [df|_{(D,S)}] \in \pi_2(\Phi_1(\mH),\Phi_1^0(\mH)) = \mZ_2.
\]
Equivalently, $\sigma(S,df)$ can be defined as the unoriented intersection number of the map $df|_D$ with the set $\Phi_1^{\mathrm{sing}}(\mH):= \Phi_1(\mH) \setminus \Phi_1^0(\mH)$ of non-surjective Fredholm operators of index one, which forms a variety of codimension two in $\Phi_1(\mH)$. As the notation suggests, $\sigma(S,df)$ does not depend on the choice of the capping disk $D$, due to the fact that $f$ is assumed to be spin. We define $\tau(f)\in \mZ_2$ to be the number modulo two of connected components $S$ of $f^{-1}(y)$ such that $\sigma(S,df)=0$. The number $\tau(f)$ is independent of the choice of the regular value $y$ and turns out to be invariant under proper Fredholm homotopies. The main result of this paper is that the invariant $\tau$, together with the space $[M,\Phi_1(\mH)]$, completely classifies the set of homotopy classes of proper Fredholm maps of index one.

\begin{theorem*}
\label{main}
Let $M$ be a simply connected Hilbert manifold. Then the proper Fredholm maps $f,g: M \rightarrow \mH$ of index one are homotopic through a proper Fredholm homotopy if and only if the following conditions are satisfied.
\begin{enumerate}[(i)]
\item The maps $df, dg: M \rightarrow \Phi_1(\mH)$ are homotopic.
\item If $f$ and $g$ are both spin, then $\tau(f)=\tau(g)$. If $f$ and $g$ are not spin, then no further condition is necessary. 
\end{enumerate}
Furthermore, the map
\[
f \mapsto \left\{ \begin{array}{ll}  ([df], \tau(f)) & \mbox{if  $f$ is spin}, \\ \; [df] & \mbox{if $f$  is not spin}, \end{array} \right.
\]
induces a bijection
\[
\mathcal{F}_1^{\mathrm{prop}}[M,\mH] \cong ( [M,\Phi_1(\mH)]_{\mathrm{sp}} \times \mZ_2) \sqcup [M,\Phi(\mH)]_{\mathrm{ns}}.
\]
\end{theorem*}

The above result should be compared to the following standard consequence of classical Pontryagin framed cobordism: If $M$ is a simply connected $(n+1)$-dimensional closed manifold with $n\geq 3$ then $[M,S^n]$ has at most two elements, and if furthermore $\pi_2(M)=0$ then $[M,S^n]\cong \mZ_2$ (see e.g.\ \cite[p.\ 185]{kos93}). The case of a general $(n+1)$-dimensional closed manifold $M$ with $n\geq 3$ is more complicated and was solved by Steenrod by introducing the algebra that is named after him, see \cite{ste47}. Recently, Konstantis~\cite{konstantis} revisited this problem and gave a geometric description of the results of Steenrod in the case in which $M$ is spin.

Let us give a closer look at the particular case $M=\mH$. Since $\mH$ is contractible, $[\mH,\Phi_1(\mH)]$ has only one class, consisting of spin maps. Then Theorem \ref{main} tells us that $\mathcal{F}_1^{\mathrm{prop}}[\mH,\mH]$ has two elements, which are distinguished by the invariant $\tau$. Let us exhibit one proper Fredholm map of index one in each of these two homotopy classes, by identifying $\mH$ with $\ell^2$, the space of square summable real sequences. The first one is the map
\[
f: \ell^2 \rightarrow \ell^2, \qquad (u_1,u_2,u_3,\dots) \mapsto ( u_1^2+u_2^2,u_3,u_4,\dots ).
\]
The map $f$ is proper, Fredholm of index one and satisfies $\tau(f)=0$ because the inverse image of the vector $-e_1 = (-1,0,0,\dots)$ is empty. In order to exhibit a map $g$ with $\tau(g)=1$, we consider the smooth proper map
\[
g_0: \mC^2 \rightarrow \mC\times \mR \cong \mR^3, \qquad (z_1,z_2) \mapsto (2 z_1 \bar{z}_2, |z_1|^2-|z_2|^2).
\]
This map sends $S^3$ to $S^2$ and restricts to the Hopf fibration on $S^3$. It lifts to a map
\[
g:   \ell^2 \rightarrow \ell^2, \qquad (u_1,u_2,u_3,\dots) \mapsto \bigl( g_0(u_1+iu_2,u_3+i u_4), u_5, u_6, \dots \bigr),
\]
which is proper and Fredholm of index one. The vector $e_3 = (0,0,1,0,\dots)$ is a regular value for $g$ and its inverse image is the circle
\[
g^{-1}(e_3) = \{ (u_1,u_2,0,0,\dots)\in \ell^2 \mid u_1^2+u_2^2=1\}.
\]
This circle can be capped by a disk that is contained in the three-sphere
\[
\{(u_1,u_2,u_3,u_4,0,0,\dots) \in \ell^2 \mid u_1^2+u_2^2+ u_3^2 + u_4^2=1\},
\]
and hence consists of regular points for $g$. This implies that $\sigma(g^{-1}(e_3),dg)=0$ and hence $\tau(g)=1$. Theorem \ref{main} implies that any proper Fredholm map of index one from $\mH$ to $\mH$ is proper Fredholm homotopic to either $f$ or $g$. In the latter case, it must be surjective.

\medskip

The paper is organized as follows. In Section \ref{topsec}, we discuss the topology of the spaces $\Phi_1(\mH)$ and $\Phi_1^0(\mH)$. In Section \ref{prelsec}, we recall the notion of framed cobordism that we introduced in \cite{ar20}, together with the main theorem of that paper and some other useful results. In Section \ref{tausec}, we rigorously define the function $\tau$ and show that it is invariant under framed cobordism. In Section \ref{modelsec}, we study three important examples, two of them being the maps $f$ and $g$ that we have introduced above. These examples will be used as normal forms in order to construct explicit framed cobordisms. Section \ref{redsec} is devoted to the reduction of some one-dimensional framed submanifolds of $\ell^2$ to the model cases introduced in Section \ref{modelsec}. In Section \ref{blockssec}, we show how to eliminate certain circles from a one-dimensional framed submanifold of $M$ and how to build framed cobordisms between simple one-dimensional framed submanifolds. Theorem \ref{main} is proved in Section \ref{proofsec}. In Appendix \ref{appsec}, we show that the space $\Phi_n^{\mathrm{sing}}(\mH)$ of non-surjective Fredholm operators of index $n\geq 0$ is stratified by finite-codimensional submanifolds and we prove a related transversality result.

\medskip

\paragraph{\sc Acknowledgments.} The research of A.\ Abbondandolo is supported by the DFG-Project 380257369 ``Morse theoretical methods in Hamiltonian dynamics''. The research of T.\ O.\ Rot is supported by NWO-NWA Startimpuls - 400.17.608.

\section{A few facts about the topology of $\Phi_1(\mH)$ and some relevant subspaces}
\label{topsec}

Let $\mH$ be a separable infinite dimensional real Hilbert space. We denote by $\Phi(\mH)$ the space of linear Fredholm operators on $\mH$ and by $\Phi_n(\mH)$ the connected component consisting of operators of index $n$, for $n\in \mZ$. Occasionally, we will need to consider the space of Fredholm operators of index $n$ from a Hilbert space $\mH_1$ to a Hilbert space $\mH_2$, and we will denote this space by $\Phi_n(\mH_1,\mH_2)$. 

The Bott periodicity theorem from \cite{bot59} can be interpreted as a computation of the homotopy groups of $\Phi_n(\mH)$, for any integer $n$, see \cite{ati89}. For $i>0$ they are given by
\[
\pi_i(\Phi_n(\mH)) = \left\{ \begin{array}{ll} \mZ & \mbox{if } i \equiv 0, \, 4 \mod 8, \\ \mZ_2 & \mbox{if } i \equiv 1, \, 2 \mod 8, \\ 0 & \mbox{if } i \equiv 3, \, 5, \, 6, \, 7 \mod 8. \end{array} \right.
\]
Each $\Phi_n(\mH)$ is the base space of a real line bundle that is known as the determinant bundle
\[
\det \rightarrow \Phi_n(\mH),
\]
whose fibers are the one-dimensional spaces
\[
\det(A) := \Lambda^{\max}(\ker A) \otimes \Lambda^{\max}(\coker A)^*, \qquad \forall A\in \Phi_n(\mH),
\]
where $\Lambda^{\max}(V)$ denotes the top degree component of the exterior algebra of the finite dimensional real vector space $V$, see \cite{qui85} and \cite{ama09}. This line bundle is non-trivial and a closed curve $A: S^1 \rightarrow \Phi_n(\mH)$ is a generator of $\pi_1(\Phi_n(\mH))=\mZ_2$ if and only if the pull-back $A^*\det$ is the non-trivial line bundle over $S^1$.

We now specialize the attention to the space $\Phi_1(\mH)$. The subset
\[
\Phi_1^0(\mH) := \{ A \in \Phi_1(\mH) \mid A \mbox{ is surjective} \}
\]
is open in $\Phi_1(\mH)$. The map
\[
p: \Phi_1^0(\mH) \rightarrow \mathrm{Gr}_1(\mH), \qquad A \mapsto \ker A,
\]
onto the Grassmannian of one-dimensional subspaces of $\mathbb{H}$ is a fiber bundle and its fibers
\[
\begin{split}
p^{-1}(L) &= \{ A \in \Phi_1^0(\mH) \mid \ker A = L \} \\ &= \{   A \in \Phi_1(\mH) \mid A|_L=0 \mbox{ and }A|_{L^{\perp}} :  L^{\perp} \rightarrow \mH \mbox{ is an isomorphism} \} \cong \mathrm{GL}(\mH)
\end{split}
\]
are contractible, thanks to Kuiper's theorem, see \cite{kui65}. Therefore, $p$ is a homotopy equivalence and $\Phi^1_0(\mH)$ has the homotopy type of $\mathrm{Gr}_1(\mH)$, that is, of $\mathrm{BO}(1)$ or $\mathbb{RP}^{\infty}$, and its only non-vanishing homotopy group is the first one:
\[
\pi_1(\Phi_1^0(\mH)) = \mZ_2.
\]
The long exact sequence in homotopy associated to the pair $(\Phi_1(\mH),\Phi_1^0(\mH))$ is
\begin{equation}
\label{exact}
\scalebox{.93}{\xymatrix@R=.3cm@C=.3cm{&0\ar@{=}[d]&\mZ_2\ar@{=}[d]& & &\mZ_2\ar@{=}[d]& & \mZ_2\ar@{=}[d]\\
\ldots\ar[r]&\pi_2(\Phi_1^0(\mH))\ar[r]&\pi_2(\Phi_1(\mH))\ar[rr]^-{\pi_2(j)}& & \pi_2(\Phi_1(\mH),\Phi_1^0(\mH))\ar[r]^-{\partial_1}&\pi_1(\Phi_1^0(\mH))\ar[rr]^{\pi_1(i)}& & \pi_1(\Phi_1(\mH))\ar[r]&\ldots}}
\end{equation}
where $j$ and $i$ denote the inclusions. The map $\pi_1(i)$ is an isomorphism as the inclusion
\[
i: \Phi_1^0(\mH) \hookrightarrow \Phi_1(\mH)
\]
pulls back the determinant line bundle over $\Phi_1(\mH)$ to the tautological bundle over $\Phi^0_1(\mH)\simeq \mathrm{BO}(1)$, which is non-trivial. It follows that the homomorphism $\partial_1$ is trivial,
\begin{equation}
\label{pi_2rel}
\pi_2( \Phi_1(\mH), \Phi_1^0(\mH)) = \mZ_2,
\end{equation}
and $\pi_2(j)$ is an isomorphism. The above exact sequence continues as follows
\[
\xymatrix@R=.3cm@C=.3cm{&\mZ_2\ar@{=}[d]& & \mZ_2\ar@{=}[d]&  & 0 \ar@{=}[d]  \\
\ldots\ar[r]&\pi_1(\Phi_1^0(\mH))\ar[rr]^{\pi_1(i)}  & &\pi_1(\Phi_1(\mH))\ar[r] & \pi_1(\Phi_1(\mH),\Phi_1^0(\mH))\ar[r]^-{\partial_0}&\pi_0(\Phi_1^0(\mH))\ar[r]& \ldots}
\]
so the fact that $\pi_1(i)$ is an isomorphism implies that
\begin{equation}
\label{pi_1rel}
\pi_1(\Phi_1(\mH),\Phi_1^0(\mH)) = 0.
\end{equation}
 
Denote by $\mD$ the unit disk in $\mC\cong \mR^2$ and by $\partial \mD\cong S^1$ its boundary. Equation (\ref{pi_2rel}) tells us that there are precisely two homotopy classes of maps from $(\mD,\partial \mD)$ to $(\Phi_1(\mH), \Phi_1^0(\mH))$. Let us explain how to distinguish them by the intersection number with the set
 \[
 \Phi^{\mathrm{sing}}_1(\mH) := \Phi_1(\mH) \setminus \Phi_1^0(\mH)
 \]
 of non-surjective Fredholm operators of index one. The proofs of the facts that we state below are standard, but for the sake of completeness are given in Appendix \ref{appsec}. The set $\Phi^{\mathrm{sing}}_1(\mH)$ is closed in $\Phi_1(\mH)$ and has the stratification
 \[
 \Phi^{\mathrm{sing}}_1(\mH) = \bigsqcup_{j\geq 1} \Phi_1^j(\mH),
 \]
 where
 \[
 \Phi_1^j(\mH) := \{ A \in \Phi_1(\mH) \mid \dim \coker A = j \}.
 \]
 The set $\Phi_1^j(\mH)$ is a submanifold of $\Phi_1(\mH)$ of codimension $j(j+1)$ and for each $h\geq 0$ the union 
 \[
 \bigcup_{j\geq h} \Phi_1^j(\mH)
 \]
 is closed in $\Phi_1(\mH)$. 
 
 Now let $\Sigma$ be a compact two-dimensional manifold with boundary and
 \[
 F: (\Sigma,\partial \Sigma) \rightarrow (\Phi_1(\mH), \Phi_1^0(\mH))
 \]
 a continuous map. Since 
 \[
 \codim \Phi_1^j(\mH) \geq 6 > \dim \Sigma \qquad \forall j\geq 2,
 \]
 we can perturb $F$ within its homotopy class and assume that $F$ is smooth, does not meet $\Phi_1^j(\mH)$ for any $j\geq 2$ and is transverse to the submanifold $\Phi_1^1(\mH)$, which has codimension two. In this case, the set
 \[
 F^{-1} ( \Phi_1^{\mathrm{sing}}(\mH)) = F^{-1} (  \Phi_1^1(\mH))
 \]
 consists of finitely many interior points of $\Sigma$ and the intersection number 
 \[
\eta(F,  \Phi_1^{\mathrm{sing}}(\mH)) \in \mZ_2
\]
is defined to be the number of these points modulo two. Note that an integer-valued intersection number cannot be defined, even if $\Sigma$ is assumed to be oriented, because the normal bundle of $\Phi_1^1(\mH)$ in $\Phi_1(\mH)$ is not orientable.

Standard arguments show that $\eta(F,  \Phi_1^{\mathrm{sing}}(\mH))$ does not depend on the smooth and transverse perturbation of $F$ and that it is a homotopy invariant, meaning that it descends to a map on
\[
\bigl[ (\Sigma,\partial \Sigma), (\Phi_1(\mH), \Phi_1^0(\mH)) \bigr].
\]
The next result says that the intersection number with $\Phi_1^{\mathrm{sing}}(\mH)$ distinguishes the elements in $\pi_2(\Phi_1(\mH),\Phi_1^0(\mH))$ and $\pi_2(\Phi_1(\mH))$.

\begin{lemma}
\label{intnumb}
In the particular case of a map 
\[
F: (\mD, \partial \mD) \rightarrow (\Phi_1(\mH), \Phi_1^0(\mH)), \qquad \mbox{resp.} \quad F: S^2 \rightarrow \Phi_1(\mH),
\]
we have that
\begin{equation}
\label{one}
\eta(F,  \Phi_1^{\mathrm{sing}}(\mH)) \in \mZ_2
\end{equation}
coincides with the element
\begin{equation}
\label{two}
 [F] \in \pi_2( \Phi_1(\mH), \Phi_1^0(\mH)) = \mZ_2, \qquad \mbox{resp.} \quad [F] \in \pi_2( \Phi_1(\mH))=\mZ_2.
\end{equation}
\end{lemma}

\begin{proof}
We first deal with the case of maps from $(\mD, \partial \mD)$ to $(\Phi_1(\mH), \Phi_1^0(\mH))$.
Since both elements are homotopy invariants, it is enough to check the equality for one map in each of the two homotopy classes in
\[
\bigl[ (\mD,\partial \mD), (\Phi_1(\mH), \Phi_1^0(\mH))\bigr] = \pi_2  (\Phi_1(\mH), \Phi_1^0(\mH)) = \mZ_2.
\]
By choosing $F$ to be a constant map into $\Phi_1^0(\mH)$, we obtain the equality of (\ref{one}) and (\ref{two}) for the trivial homotopy class. It suffices then to exhibit a map
\[
F: (\mD,\partial \mD) \rightarrow  (\Phi_1(\mH), \Phi_1^0(\mH))
\]
such that
\[
\eta(F, \Phi_1^{\mathrm{sing}}(\mH)) = 1,
\]
because by the homotopy invariance of the intersection number, $[F]$ must be the non-trivial element in $ \pi_2  (\Phi_1(\mH), \Phi_1^0(\mH))$. The existence of such a map $F$ is established in the example below, which will be useful also later on. 

The case of a map from $S^2$ to $\Phi_1(\mH)$ follows from the previous case: Also here, it is enough to show the existence of a map $F': S^2 \rightarrow \Phi_1(\mH)$ with $\eta(F',\Phi_1(\mH))=1$. Let $F: (\mD,\partial \mD) \rightarrow   (\Phi_1(\mH), \Phi_1^0(\mH))$ be such that $\eta(F, \Phi_1^{\mathrm{sing}}(\mH)) = 1$. Since the boundary homomorphism $\partial_1$ in (\ref{exact}) vanishes, the loop $F|_{\partial \mD}$ can be capped by a disk in $\Phi_1^0(\mH)$, and gluing this disk to $F$ we obtain a map $F': S^2 \rightarrow \Phi^1(\mH)$ with $\eta(F', \Phi_1^{\mathrm{sing}}(\mH)) = 1$.
\end{proof}

\begin{example}
\label{exsigma1}
Identify $\mH$ with the Hilbert space $\ell^2$ of square summable real sequences. Consider the smooth map
\[
F: \mD \rightarrow \Phi_1(\ell^2)
\]
that is defined as follows: For every $(x,y)\in \mD$ we define $F(x,y)$ to be the Fredholm operator of index one
\[
F(x,y): (v_1,v_2,v_3, \dots) \mapsto (xv_1 + yv_2, v_3, v_4, \dots ).
\]
The kernel of $F(x,y)$ is one-dimensional for every $(x,y)\neq (0,0)$, whereas
\begin{equation}
\label{atzero}
F(0,0) : (v_1,v_2,v_3,\dots) \mapsto (0,v_3,v_4, \dots)
\end{equation}
has a two-dimensional kernel. Therefore, $F$ meets $\Phi_1^{\mathrm{sing}}(\ell^2)$ only at $(0,0)$, and $F(0,0)$ belongs to $\Phi_1^1(\ell^2)$. Let us check that $F$ meets  $\Phi_1^1(\ell^2)$ transversally at $(0,0)$. By (\ref{atzero}) and the formula given in Proposition \ref{singular} in Appendix \ref{appsec}, the tangent space of $\Phi_1^1(\ell^2)$ at $F(0,0)$ is
\begin{equation}
\label{tangent}
T_{F(0,0)} \Phi_1^1(\ell^2) = \{ A \in \mathrm{L}(\ell^2) \mid \langle A e_1,e_1 \rangle = \langle A e_2,e_1 \rangle = 0\},
\end{equation}
where $\mathrm{L}(\ell^2)$ denotes the space of bounded linear operators on $\ell^2$ and $\langle\cdot,\cdot \rangle$ the scalar product of $\ell^2$. The image of the differential of $F$ at $(0,0)$ is the two-dimensional plane spanned by the operators
\[
\begin{split}
\partial_x F(0,0) &: (v_1,v_2,v_3, \dots) \mapsto (v_1,v_3,v_4, \dots) , \\  \partial_y F(0,0) &: (v_1,v_2,v_3, \dots) \mapsto (v_2,v_3,v_4, \dots),
\end{split}
\]
and one readily checks that this plane has trivial intersection with the subspace (\ref{tangent}). Therefore, $F$ meets $\Phi_1^1(\ell^2)$ transversally at $(0,0)$ and
\[
\eta(F, \Phi_1^{\mathrm{sing}}(\mH)) = [F] = 1.
\]
\end{example}

\begin{remark}
\label{hom-rel-bdry}
Consider two maps
\[
F_0,F_1: (\mD, \partial \mD) \rightarrow (\Phi_1(\mH), \Phi_1^0(\mH)),
\]
such that $F_0|_{\partial \mD} = F_1|_{\partial \mD}$. The maps $F_0$ and $F_1$ are homotopic through a homotopy $H$ such that $H(t,\cdot) = F_0 = F_1$ on $\partial \mD$ for every $t\in [0,1]$ if and only if the intersection numbers of $F_0$ and $F_1$ with $\Phi_1^{\mathrm{sing}}(\mH)$ coincide. Indeed, $F_0$ and $F_1$ can be glued along the boundary and produce the map $F:S^2\rightarrow \Phi_1(\mH)$ given by
\[
    F(x,y,z):=\begin{cases} F_0(x,y) \quad &\mbox{if }z\geq 0,\\
      F_1(x,y) &\mbox{if }z\leq 0,
      \end{cases}
 \]
which satisfies
\[
\eta(F,\Phi_1^{\mathrm{sing}}(\mH)) = \eta(F_0,\Phi_1^{\mathrm{sing}}(\mH)) + \eta(F_1,\Phi_1^{\mathrm{sing}}(\mH)).
\]
By Lemma \ref{intnumb}, the map $F$ is nullhomotopic if and only if the intersection numbers of $F_0$ and $F_1$ with $\Phi_1^{\mathrm{sing}}(\mH)$ coincide. In this case, $F$ extends over the ball and the map 
\[
H:[0,1]\times\mD\rightarrow \Phi_1(\mH), \qquad (t,x,y) \mapsto F\bigl(x,y,(1-2t)\sqrt{1-x^2-y^2}\bigr),
\]
is a homotopy from $F_0$ to $F_1$ such that $H(t,\cdot)=F_0=F_1$  on $\partial \mD$ for every $t\in [0,1]$.
\end{remark}

We conclude this section by the following result about maps on the annulus.

\begin{lemma}
\label{annulus}
Let $\Sigma:= [0,1]\times S^1$ and 
\[
F_0,F_1 : (\Sigma,\partial \Sigma) \rightarrow (\Phi_1(\mH), \Phi_1^0(\mH))
\]
be maps such that the loops $F_0(0,\cdot)$, $F_0(1,\cdot)$, $F_1(0,\cdot)$ and $F_1(1,\cdot)$ are null-homotopic in $\Phi_1^0(\mH)$. Then $F_0$ and $F_1$ are homotopic if and only if
\begin{equation}
\label{equal}
\eta(F_0, \Phi_1^{\mathrm{sing}}(\mH)) = \eta(F_1, \Phi_1^{\mathrm{sing}}(\mH)).
\end{equation}
\end{lemma}

\begin{proof}
The necessity of condition (\ref{equal}) is clear because the intersection number is a homotopy invariant. Now assume that (\ref{equal}) holds. Consider the map $F_0$. Thanks to the condition on the boundary loops, up to a first homotopy we can assume that
\[
F_0(0,z) = A_0 \qquad \mbox{and} \qquad F_0(1,z) = A_1 \qquad  \forall z\in S^1,
\]
for some given $A_0,A_1\in \Phi_1^0(\mH)$. Fix some $z_0\in S^1$. Thanks to (\ref{pi_1rel}), up to a second homotopy we can assume that $F_0(s,z_0)$ belongs to $\Phi_1^0(\mH)$ for every $s\in [0,1]$. Set 
\[
\Gamma:=  \bigl( \{0,1\} \times S^1\bigr) \cup \bigl( [0,1]\times \{z_0\} \bigr).
\]
It is easy to construct a homotopy
\[
H: [0,1] \times \Gamma \rightarrow \Phi_1^0(\mH) 
\]
such that $H(0,\cdot) = F_0(\cdot)$ and $H(1,\cdot)$ is constantly equal to $A_0$: Just define 
\[
H(t,s,z_0) := F_0((1-t)s,z_0), \quad H(t,0,z) := A_0, \quad H(t,1,z):= F_0(1-t,z_0), 
\]
for every $(t,s)\in [0,1]^2$ and $z\in S^1$. Setting $H(0,\cdot)=H_0$ on the whole $\Sigma$ and extending $H$ to $[0,1]\times \Sigma$ we obtain a homotopy
\[
H: ([0,1] \times \Sigma, [0,1]\times \partial \Sigma) \rightarrow (\Phi_1(\mH), \Phi_1^0(\mH))
\]
that connects $F_0$ to a map that is constantly equal to $A_0$ on $\Gamma$. Up to replacing $F_0$ by $H(1,\cdot)$, we may assume that $F_0$ has the latter property. Similarly, we may assume that also $F_1$ is constantly equal to $A_0$ on $\Gamma$. The quotient space $\Sigma/\Gamma$ is a two-sphere $S^2$. Denote by
\[
p: \Sigma \rightarrow S^2 
\]
the quotient map and by $\gamma\in S^2$ the image of $\Gamma$ by $p$. Being constant on $\Gamma$, both $F_0$ and $F_1$ factorize through $p$:
\[
F_0 = \tilde{F}_0 \circ p \qquad \mbox{and} \qquad F_1 = \tilde{F}_1 \circ p
\]
for some maps $\tilde{F}_0, \tilde{F}_1 : S^2 \rightarrow \Phi_1(\mH)$ mapping $\gamma$ into $A_0$. By construction
\[
\eta(\tilde{F}_0,\Phi_1^{\mathrm{sing}}(\mH)) = \eta(F_0,\Phi_1^{\mathrm{sing}}(\mH)) \quad \mbox{and} \quad \eta(\tilde{F}_0,\Phi_1^{\mathrm{sing}}(\mH)) = \eta(F_0,\Phi_1^{\mathrm{sing}}(\mH)).
\]
Therefore, assumption (\ref{equal}) and Lemma \ref{intnumb} imply that $\tilde{F}_0$ and $\tilde{F}_1$ are homotopic through a homotopy that is constantly equal to $A_0$ on $[0,1]\times \{\gamma\}$. By right composition with $p$ we obtain the desired homotopy between $F_0$ and $F_1$.
\end{proof}

\begin{remark}
Note that, unlike the case of maps from the disk discussed in Remark \ref{hom-rel-bdry}, the intersection number with $\Phi_1^{\mathrm{sing}}(\mH)$ is not the only obstruction for two maps
\[
F_0, F_1: (\Sigma,\partial \Sigma) \rightarrow (\Phi_1(\mH),\Phi_1^0(\mH))
\]
satisfying the assumptions of Lemma \ref{annulus} and $F_0|_{\partial \Sigma}= F_1|_{\partial \Sigma}$ to be homotopic through a homotopy fixing the boundary. Indeed, the fact that $\pi_1(\Phi_1^0(\mH)) \cong \pi_1(\Phi_1(\mH))$ is non-trivial produces a second obstruction: Consider a path
\[
F: [0,1] \rightarrow \Phi_1^0(\mH)
\]
satisfying $F(0)=F(1)$ that is not contractible with fixed ends in $\Phi_1^0(\mH)$ and define the maps $F_0$ and $F_1$ to be
\[
F_0(x,y) := F(0) = F(1) , \qquad F_1(x,y) = F(x) \qquad \forall (x,y) \in \Sigma = [0,1]\times S^1.
\]
These maps coincide on $\partial \Sigma$, satisfy the assumptions of Lemma \ref{annulus} and have both intersection number zero with  $\Phi_1^{\mathrm{sing}}(\mH)$. However they are not homotopic through a homotopy $H$ such that $H(t,\cdot) = F_0=F_1$ on $\partial \Sigma$ for every $t\in [0,1]$.
\end{remark} 

\section{Framed cobordism on Hilbert manifolds}
\label{prelsec}

In this section, we recall some of the main definitions and results from \cite{ar20} and we add a few notions and statements that will be useful in the following sections.
By a Hilbert manifold we mean a connected paracompact smooth manifold modeled on the Hilbert space $\mH$. A smooth map $f: M \rightarrow N$ between the Hilbert manifolds $M$ and $N$ is said to be Fredholm of index $n$ if its differential at every point is a Fredholm operator of index $n$. Since the general linear group of $\mH$ is contractible, the tangent bundles of $M$ and $N$ are trivial. By fixing trivializations of these bundles, we can see the differential of the Fredholm map $f: M \rightarrow N$ of index $n$ as a map
\[
df: M \rightarrow \Phi_n(\mH).
\]
The map $f: M \rightarrow N$ is said to be proper if $f^{-1}(K)$ is compact for any $K\subset N$ compact. The symbol $\mathcal{F}_n^{\mathrm{prop}}[M,N]$ denotes the space of equivalence classes of proper Fredholm maps from $M$ to $N$ modulo proper Fredholm homotopies. Note that a Fredholm homotopy $h: [0,1] \times M \rightarrow N$ between two Fredholm maps of index $n$ has index $n+1$. Note also that the properness assumption on $h$ is stronger than asking each map $h(t,\cdot): M \rightarrow N$ to be a proper map (a good example is the map $h: [0,1]\times \mR \rightarrow \mR$ given by $h(t,x) = tx^2+x$, which is not proper but restricts to a proper map on $\{t\} \times \mR$ for every $t\in [0,1]$).

From now on, we fix a trivialization of the tangent bundle of the Hilbert manifold $M$ and choose the target manifold $N$ to be the model Hilbert space $\mH$. All Fredholm maps and Fredholm homotopies are assumed to be smooth. In the following definition, we see the empty set as a submanifold of $M$ of arbitrary dimension $n\in \mZ$. If $n<0$, this is the only submanifold of dimension $n$.

\begin{definition}
Let $n\in \mZ$ and $X \subset M$ be a compact submanifold of dimension $n$. A framing of $X$ is a continuous map
\[
A: M \rightarrow \Phi_n(\mH)
\]
such that $\ker A(x)=T_x X$ for every $x\in X$. The pair $(X,A)$ is called framed submanifold of dimension $n$ of $M$.
\end{definition}

Framed submanifolds are always assumed to be compact.
Note that, unlike in the finite dimensional Pontryagin framed cobordism theory, here framings are assumed to be defined on the whole $M$ and not just on the submanifold $X$. The reason for this is that the space of Fredholm operators $\Phi_n(\mH)$ has a non-trivial topology, so maps from $X$ to $\Phi_n(\mH)$ do not necessarily extend to $M$.  

 If $f: M \rightarrow \mH$ is a proper Fredholm map of index $n$ and $y\in \mH$ is a regular value of $f$, then the pair
\[
(f^{-1}(y),df)
\]
is an $n$-dimensional framed submanifold of $M$. Note that the regular values of $f$ form a dense subset of $\mH$ by the Sard-Smale theorem, see \cite{sma65}. The above pair is called Pontryagin framed manifold of $f$.

We denote by 
\[
\begin{split}
S_L&: \mR \oplus \mH \rightarrow \mH, \qquad (s,v) \mapsto v, \\
S_R &: \mH \rightarrow \mR \oplus \mH, \qquad v \mapsto (0,v),
\end{split}
\]
the left and right shifts, which are Fredholm operators of index $1$ and $-1$, respectively. Their action by composition defines maps
\[
\begin{split}
S_L^* &: \Phi_n(\mH) \rightarrow \Phi_{n+1}(\mR \oplus \mH,\mH), \qquad S_L^* A := A S_L, \\
S_R^*&:  \Phi_{n+1}(\mR \oplus \mH,\mH) \rightarrow  \Phi_n(\mH), \qquad S_R^* A := A S_R,
\end{split}
\]
such that $S_R^* S_L^*= (S_LS_R)^*$ is the identity on $ \Phi_n(\mH)$.

\begin{definition} 
A cobordism between the compact $n$-dimensional submanifolds $X_0$ and $X_1$ of $M$ is a compact submanifold with boundary $W\subset [0,1]\times M$ of dimension $n+1$ such that
\[
\partial W \subset \{0,1\} \times M, \qquad W \cap ([0,\epsilon) \times M) = [0,\epsilon) \times X_0, \qquad W \cap ((1-\epsilon,1]\times M) = (1-\epsilon,1] \times X_0,
\]
for some $\epsilon \in (0,\frac{1}{2})$. A framing of the cobordism $W$ is a continuous map
\[
B: [0,1] \times M \rightarrow \Phi_{n+1}(\mR\oplus \mH, \mH)
\]
such that
\[
\ker B(t,x) = T_{(t,x)} W \qquad \forall (t,x)\in W,
\]
and for every $t_0\in [0,\epsilon)$, $t_1\in (1-\epsilon,1]$ and $x\in M$ we have
\[
B(t_0,x) = S_L^* A_0(x) \qquad B(t_1,x) = S_L^* A_1(x),
\]
where $A_0$ and $A_1$ are framings of $X_0$ and $X_1$. In this case, the pair $(W,B)$ is called a framed cobordism from $(X_0,A_0)$ to $(X_1,A_1)$.
\end{definition}

\begin{remark}
In \cite{ar20} we actually required framings of $X\subset M$ or $W\subset [0,1]\times M$ to be smooth on $X$ and $W$, respectively. This requirement can be dropped, because any continuous framing can be smoothened.
\end{remark}

\begin{remark}
\label{trivial1}
Note that if two $n$-dimensional framed submanifolds $(X_0,A_0)$ and $(X_1,A_1)$ are framed cobordant, then the maps $A_0,A_1: M \rightarrow \Phi_n(\mH)$ are homotopic: If $(W,B)$ is a framed cobordism from $(X_0,A_0)$ to $(X_1,A_1)$ then $S_R^*B$ is a homotopy from $A_0$ to $A_1$.
\end{remark}

\begin{remark}
\label{trivial2}
 If $A_0$ and $A_1$ are framings of the same compact submanifold $X\subset M$, then a homotopy
\[
H: [0,1] \times M \rightarrow \Phi_n(\mH)
\]
such that
\begin{equation}
\label{kernel}
\ker H(t,x) = T_x X \qquad \forall x\in X
\end{equation}
defines a framing of the trivial cobordism $[0,1]\times X$: We define $B(t,x) := S_L^* H(\chi(t),x)$, where $\chi: [0,1] \rightarrow [0,1]$ is a continuous function such that $\chi=0$ in a neighborhood of $0$ and $\chi=1$ in a neighborhood of $1$. Therefore, if $A_0$ and $A_1$ are homotopic through a homotopy $H$ that satisfies (\ref{kernel}), then $(X,A_0)$ and $(X,A_1)$ are framed cobordant.
\end{remark}

Framed cobordism induces an equivalence relation on the set of $n$-dimensional framed submanifolds of $M$. The set of equivalence classes of this relation is denoted by
\[
\Omega_n^{\mathrm{fr}}(M).
\]
If $y_1$ and $y_2$ are regular values of $f: M \rightarrow \mH$, the corresponding Pontryagin framed manifolds $(f^{-1}(y_1),df)$ and $(f^{-1}(y_2),df)$ are framed cobordant. More generally, if $f_1:M \rightarrow \mH$ and $f_2:M \rightarrow \mH$ are proper Fredholm homotopic with regular values $y_1$ and $y_2$, then the corresponding Pontryagin framed manifolds $(f^{-1}_1(y_1),df)$ and $(f^{-1}_2(y_2),df)$ are framed cobordant. Therefore, the Pontryagin construction induces a map
\[
\mathcal{F}_n^{\mathrm{prop}}[M,\mH] \rightarrow \Omega_n^{\mathrm{fr}} (M).
\]
One of the main results of \cite{ar20} is that this map is bijective.

\begin{theorem}{\em (\cite[Theorem 7.1]{ar20})}
\label{fromold}
The map 
\[
\mathcal{F}_n^{\mathrm{prop}}[M,\mH] \rightarrow \Omega_n^{\mathrm{fr}} (M)
\]
is bijective.
\end{theorem}

Therefore, the classification of proper Fredholm maps from $M$ to $\mH$ modulo proper Fredholm homotopy is reduced to the classification of framed submanifolds of $M$ modulo framed cobordism. The classification of the latter objects was carried out in \cite{ar20} for all indices $n\leq 0$, see \cite[Theorems 3 and 4]{ar20}.

The following easy result will be useful in order to construct framed cobordisms.

\begin{lemma}
\label{functors}
Let $(X,A)$ be a framed submanifold of the Hilbert manifold $M$ and $\varphi: M \rightarrow M$ a diffeomorphism that is smoothly isotopic to the identity. Then $(X,A)$ is framed cobordant to $(\varphi(X),A\circ \varphi^{-1} d\varphi^{-1})$.
\end{lemma}

\begin{proof}
Let $\psi: [0,1]\times M \rightarrow M$ be a smooth map such that $\psi(t,\cdot)$ is a diffeomorphism for every $t$, $\psi(t,\cdot) = \mathrm{id}$ for every $t\in [0,\frac{1}{3}]$ and $\psi(t,\cdot) = \varphi$ for every $t\in [\frac{2}{3},1]$. Then the map
\[
\Psi: [0,1] \times M \rightarrow [0,1] \times M, \qquad (t,x) \mapsto (t,\psi(t,x)),
\]
is a diffeomorphism. Let $(W_0,B_0)$ be the trivial framed cobordism
\[
W_0:= [0,1]\times X, \qquad B_0(t,x) := S_L^* A(x) \qquad \forall (t,x)\in [0,1] \times M.
\]
Then the pair $(W,B)$ that is defined as
\[
W:= \Psi(W_0), \qquad B:= \Psi_* B_0 := B_0 \circ \Psi^{-1} d\Psi^{-1}
\]
is readily seen to be a framed cobordism from $(X,A)$ to  $(\varphi(X),A\circ \varphi^{-1} d\varphi^{-1})$.
\end{proof}

Assume that the Hilbert manifold $M$ is simply connected. In this case, framed submanifolds of $M$ are automatically orientable manifolds, see \cite[Remark 4.2]{ar20}. Actually, a framing $A: M \rightarrow \Phi_n(\mH)$ of the $n$-dimensional compact submanifold $X$ allows us to compare the orientations of the different connected components of $X$. Indeed, if $X_0$ and $X_1$ are connected components of $X$ and $\gamma:[0,1] \rightarrow M$ is a path such that $\gamma(0)\in X_0$ and $\gamma(1)\in X_1$, we can consider the line bundle
\[
(A\circ \gamma)^* \det \rightarrow [0,1],
\]
whose fiber at $0$, resp.\ $1$, is $\Lambda^n(T_{\gamma(0)} X_0)$, resp.\ $\Lambda^n(T_{\gamma(1)} X_1)$. We shall say that two orientations of $X_0$ and $X_1$, or equivalently of the one-dimensional spaces $\Lambda^n(T_{\gamma(0)} X_0)$ and $\Lambda^n(T_{\gamma(1)} X_1)$, are $A$-coherent if they extend to an orientation of the above trivial line bundle. The fact that $M$ is simply connected implies that this notion does not depend on the choice of the path connecting $X_0$ to $X_1$.

\section{The invariant $\tau$}
\label{tausec}

Throughout this section, we assume the Hilbert manifold $M$ to be simply connected. 

\begin{definition}
A continuous map $A: M \rightarrow \Phi_1(\mH)$ is said to be spin if the homomorphism
\[
\pi_2(A) : \pi_2(M) \rightarrow \pi_2 ( \Phi_1(\mH)) = \mZ_2
\]
is zero. A Fredholm map $f: M \rightarrow \mH$ of index one is said to be spin if its differential $df: M \rightarrow \Phi_1(\mH)$ is spin.
\end{definition}

\begin{remark}
\label{remspin}
Let us comment on our choice of using the term ``spin'' in this context. Let $N$ be a simply connected finite dimensional manifold. Recall that $N$ is said to be spin if the  second Whitney class $w_2(TN)$ of its tangent bundle vanishes. We can detect whether $N$ is spin from the differential of a map $f:N\rightarrow \mR^n$ as follows. 
 If considered separately,  $\ker df$ and $\coker df$ need not be vector bundles, but the $K$-theory class $\ker df- \coker df$ is well-defined (see \cite{ati89}). As $f$ is homotopic to a constant map, this $K$-theory class equals $[TN]-[\mR^n]$, where $\mR^n$ denotes the trivial rank $n$ bundle over $N$. This $K$-theory is classified by a map $\mu:N\rightarrow BO$. As $N$ is assumed to be simply connected, the second Stiefel-Whitney class $w_2(TN)$ vanishes if and only if $\pi_2(\mu):\pi_2(N)\rightarrow \pi_2(BO)=\mathbb{Z}_2$ is trivial, i.e.~the manifold $N$ is spin if and only if $\pi_2(\mu)=0$. Now let $M:=N\times \mH$ and define the Fredholm map $g:M\rightarrow \mR^n\times\mH$ by $g(x,y):=(f(x),y)$. By trivializing the tangent bundle of $M$ and identifying $\mR^n\times \mH$ with $\mH$, the differential of $g$ defines a map $dg:M\rightarrow \Phi(\mH)$. Now $\pi_2(dg)=0$ if and only if $\pi_2(\mu)=0$, hence $g$ is a spin Fredholm map if and only if $N$ is a spin manifold.
\end{remark}

That a map $M\rightarrow \Phi_1(\mH)$ is spin or not depends only on its homotopy class. Therefore, the set of homotopy classes of maps from $M$ into $\Phi_1(\mH)$ has the partition
\[
[ M, \Phi_1(\mH) ] = [ M, \Phi_1(\mH) ]_{\mathrm{sp}} \sqcup [ M, \Phi_1(\mH) ]_{\mathrm{ns}},
\]
where the first set denotes the set of spin homotopy classes and the second one the set of non-spin ones. If the one-dimensional framed submanifolds $(X_0,A_0)$ and $(X_1,A_1)$ are framed cobordant, then $A_0$ and $A_1$ are homotopic and hence are either both spin or both non-spin.

Let $(X,A)$ be a one-dimensional framed submanifold of $M$ and assume $A$ to be spin. Let $S\cong S^1$ be a connected component of $X$ and let $\varphi: \mD \rightarrow M$ be a continuous map such that $\varphi|_{\partial \mD}$ is a homeomorphism onto $S$. Since $A(x)$ belongs to $\Phi_1^0(\mH)$ for every $x\in S$, the composition $A\circ \varphi$ gives us a map
\begin{equation}
\label{phiA}
A\circ \varphi: (\mD,\partial \mD) \rightarrow (\Phi_1(\mH), \Phi_1^0(\mH)) ,
\end{equation}
and hence an element
\[
[A\circ \varphi] \in \pi_2 ( \Phi_1(\mH), \Phi_1^0(\mH)) = \mZ_2.
\]

\begin{lemma}
If $A$ is spin, then the element $[A\circ \varphi]\in \mZ_2$ is independent of the choice of the map $\varphi$.
\end{lemma}

\begin{proof}
Denote by $\widehat{\varphi}$ the map $\widehat{\varphi}(z)=\varphi(\bar{z})$, where $\bar{z}$ is the complex conjugate of $z$. Since $\mZ_2$ has only two elements and $[A\circ \varphi]$ is zero if and only if the map (\ref{phiA}) is homotopic to a map taking values into $\Phi_1^0(\mH)$, we have
\[
[A\circ \varphi] = [A\circ \widehat{\varphi}].
\]
Now let $\psi: \mD \rightarrow M$ be another map mapping $\partial \mD$ homeomorphically onto $S$. Up to the possible replacement of $\varphi$ by $\widehat{\varphi}$, we may assume that the homeomorphisms
\[
\psi|_{\partial \mD}: \partial \mD \rightarrow S \qquad \mbox{and} \qquad \varphi|_{\partial \mD}: \partial \mD \rightarrow S
\]
are isotopic. Therefore, we can modify $\psi$ without affecting $[A\circ \psi]$ so that $\psi|_{\partial \mD} = \varphi|_{\partial \mD}$. By gluing the maps $\psi$ and $\varphi$ along the boundary of the disk we obtain a map
\[
\psi \# \varphi: S^2 \rightarrow M.
\]
The spin assumption on $A$ guarantees that
\[
\pi_2 ( A\circ (\psi \# \varphi)) : \pi_2(S^2) \rightarrow \pi_2(\Phi_1(\mH))
\]
is trivial and this gives us a homotopy from $A\circ \varphi$ to $A\circ \psi$ mapping $[0,1]\times \partial \mD$ into $\Phi_1^0(\mH)$. Therefore, $[A\circ \varphi] = [A\circ \psi]$.
\end{proof}

Thanks to the above lemma, we can give the following definition.

\begin{definition}
\label{sigmadef}
Let $(X,A)$ be a one-dimensional framed submanifold of the simply connected Hilbert manifold $M$ with $A: M \rightarrow \Phi_1(\mH)$ spin. For every connected component $S$ of $X$ we define 
\[
\sigma(S,A)\in \mZ_2
\]
to be the element
\[
[A\circ \varphi] \in \pi_2 (\Phi_1(\mH), \Phi_1^0(\mH)) = \mZ_2,
\]
where $\varphi: \mD \rightarrow M$ maps $\partial \mD$ homeomorphically onto $S$. 
\end{definition}

Thanks to Lemma \ref{intnumb}, $\sigma(S,A)$ can be expressed in terms of the intersection number with the space of non-surjective Fredholm operators of index one.

\begin{lemma}
If $\varphi: \mD \rightarrow M$ is a continuous map sending $\partial\mD$ homeomorphically onto a component $S$ of the one-dimensional framed submanifold $(X,A)$ of $M$, with $A$ spin, then
\[
\sigma(S,A) = \eta( A\circ \varphi, \Phi^{\mathrm{sing}}_1(\mH)).
\]
\end{lemma}

The function $\sigma$ induces the following function $\tau$ on the set of submanifolds equipped with a spin framing.

\begin{definition}
In the setting of Definition \ref{sigmadef}, we define $\tau(X,A)\in \mZ_2$ to be the parity of the set of connected components $S$ of $X$ such that $\sigma(S,A)=0$.
\end{definition}

In the next proposition, we prove that $\tau$ is a framed cobordism invariant.

\begin{proposition}
\label{tauprop}
Let $(X_0,A_0)$ and $(X_1,A_1)$ be one-dimensional framed submanifolds of the simply connected Hilbert manifold $M$. If $(X_0,A_0)$ and $(X_1,A_1)$ are framed cobordant and the maps $A_0$ and $A_1$ are spin then
\[
\tau(X_0,A_0) = \tau(X_1,A_1).
\]
\end{proposition}

\begin{proof}
Denote by $(W,B)$ a framed cobordism from $(X_0,A_0)$ to $(X_1,A_1)$. Then $W$ is a compact two-dimensional submanifold of $[0,1]\times M$ and is orientable, thanks to the fact that $[0,1]\times M$ is simply connected (see \cite[Remark 4.2]{ar20}).
Up to an ambient isotopy we may assume that the function
\[
[0,1]\times M \rightarrow \mR, \qquad (t,x) \mapsto t,
\]
restricts to a Morse function on $W$, which we denote by $\mu$. By standard facts about height-functions, this is equivalent to the fact that the map
\[
W \rightarrow \mathrm{Gr}_2(\mR \oplus \mH), \qquad (t,x) \mapsto T_{(t,x)} W,
\]
is transverse to the Banach submanifold $\mathrm{Gr}_2(\mH)$ of $ \mathrm{Gr}_2(\mR \oplus \mH)$ consisting of 2-planes that are contained in $(0)\oplus \mH$ (as usual, we are identifying the tangent spaces to $M$ with $\mH$ by means of the fixed trivialization of $TM$). The set  $\mathrm{crit}\, \mu$ of critical points of $\mu$ is precisely the inverse image of $\mathrm{Gr}_2(\mH)$ by the above map. Since $B$ is a framing of $W$, the above map agrees with
\[
W \rightarrow \mathrm{Gr}_2(\mR \oplus \mH), \qquad (t,x) \mapsto \ker B(t,x).
\]
From Proposition \ref{apptrans} in Appendix \ref{appsec} we deduce that the critical points of $\mu$ are precisely the points $(t,x)\in W$ at which the map
\[
S_R^* B: W \rightarrow \Phi_1(\mH)
\]
intersects $\Phi^{\mathrm{sing}}_1(\mH)$, and that the intersection occurs at $\Phi_1^1(\mH)$ and is transverse.
We conclude that
\[
\eta(S_R^*B|_W, \Phi^{\mathrm{sing}}_1(\mH)) = \# \mathrm{crit}\, \mu \mod 2.
\]
Morse theory on the orientable surface with boundary $W$ implies that the latter number coincides modulo two with the number of boundary components of $W$ and hence
\[
\eta(S_R^*B|_W, \Phi^{\mathrm{sing}}_1(\mH)) = |X_0| + |X_1| \mod 2,
\]
where $|X_0|$ and $|X_1|$ denote the number of connected components of $X_0$ and $X_1$.
Let $\Sigma$ be the closed surface that is obtained from $W$ by capping all the boundary circles by disks. Choose a capping disk $\varphi: \mD \rightarrow M$ for each connected component of $X_0$ and for each connected component of $X_1$. By using the maps $\varphi$, we can extend the embedding $W\hookrightarrow [0,1]\times M$ to a continuous map 
\[
\psi: \Sigma \rightarrow [0,1] \times M
\]
and we define the map
\[
F : \Sigma \rightarrow \Phi_1(\mH), \qquad F := S_R^*B\circ \psi.
\]
We now consider the intersection number of $F$ with $\Phi^{\mathrm{sing}}_1(\mH)$. Since the capping disk corresponding to the component $S$ of $X_0$, respectively of $X_1$, contributes by $\sigma(S,A_0)$, respectively $\sigma(S,A_1)$, we obtain
\[
\begin{split}
\eta(F, \Phi^{\mathrm{sing}}_1(\mH)) &= \eta(S_R^*B|_W, \Phi^{\mathrm{sing}}_1(\mH)) + \sum_S \sigma(S,A_0) + \sum_S\sigma(S,A_1) \\ &= |X_0| + |X_1| + \sum_S \sigma(S,A_0) + \sum_S\sigma(S,A_1),\end{split}
\]
where the first sum ranges over the connected components of $X_0$ and the second one over those of $X_1$. Note that 
\[
\begin{split}
|X_0| + \sum_S \sigma(S,A_0) &= \tau(X_0,A_0) \mod 2, \\
|X_1| + \sum_S \sigma(S,A_1) &= \tau(X_1,A_1) \mod 2,
\end{split}
\]
and hence the previous identity reads
\begin{equation}
\label{deves0}
\eta(F, \Phi^{\mathrm{sing}}_1(\mH)) = \tau(X_0,A_0) + \tau(X_1,A_1).
\end{equation}

The surface $\Sigma$ has the structure of a two-dimensional CW-complex, which we can choose to have only one two-cell for each connected component of $\Sigma$. Since the space $[0,1]\times M$ is simply connected, $\psi$ is homotopic to a map $\psi'$ that is constant on the one-skeleton of each connected component of $\Sigma$. Collapsing the one-skeleton of each connected component of $\Sigma$ produces a disjoint union of spheres, which we denote by $\Sigma'$, and the map $\psi'$ factorizes through it: 
\[
\psi': \Sigma \stackrel{\psi_1'}{\longrightarrow} \Sigma' \stackrel{\psi'_2}{\longrightarrow} [0,1]\times M.
\]
Therefore, the map $F= S_R^* B \circ \psi$ is homotopic to a map $F'$ that factorizes as
\[
F': \Sigma \stackrel{\psi_1'}{\longrightarrow} \Sigma' \stackrel{\psi_2'}{\longrightarrow} [0,1]\times M \stackrel{S_R^*B}{\longrightarrow} \Phi_1(\mH).
\]
By the orientability assumption on $A_0$ and $A_1$, we have that $\pi_2(S^*_RB)=0$, and hence $S_R^* B\circ \psi_2'$ is homotopic to a constant map. We conclude that $F'$ is homotopic to a constant map, and so is $F$. From this it follows that
\[
\eta(F, \Phi^{\mathrm{sing}}_1(\mH)) = 0,
\]
and identity (\ref{deves0}) implies the equality of $\tau(X_0,A_0)$ and $\tau(X_1,A_1)$,
\end{proof}

Thanks to the above result and to the fact that the Pontryagin manifold is uniquely defined up to framed cobordism, we can give the following definition.

\begin{definition}
Let $f: M \rightarrow \mH$ be a spin proper Fredholm map of index one on a simply connected Hilbert manifold $M$. Then $\tau(f)\in \mZ_2$ is defined as $\tau(f^{-1}(y),df)$, where $y\in \mH$ is any regular value of $f$.
\end{definition}

The number $\tau(f)$ is invariant under proper Fredholm homotopies.

\section{The models}
\label{modelsec}

In this section, we study three particular one-dimensional framed submanifolds of the Hilbert space $\ell^2$ of square summable real sequences which will serve as models for more general one-dimensional framed submanifolds. The standard basis of $\ell^2$ is denoted by $e_1,e_2,e_3,\dots$. 

\begin{example}
\label{exf}
The first example is induced by the proper smooth function
\[
f_0: \mR^2 \rightarrow \mR, \qquad f_0(x,y) := x^2+y^2,
\]
which lifts to the proper Fredholm map of index one
\[
f: \ell^2 \rightarrow \ell^2, \qquad (u_1,u_2,u_3,\dots) \mapsto (f_0(u_1,u_2),u_3, u_4, \dots).
\]
The vector $e_1$ is a regular value of $f$ and its inverse image is the circle
\begin{equation}
\label{defnS0}
S_0 := f^{-1}(e_1) =  \{  (u_1,u_2,0,0,,\dots)\in \ell^2  \mid u_1^2 + u_2^2 = 1\}.
\end{equation}
Therefore, $(S_0,df)$ is a one-dimensional framed submanifold of $\ell^2$. Note that, up to the multiplication factor 2, the restriction of $df: \ell^2 \rightarrow \Phi_1(\mH)$ to the disk
\begin{equation}
\label{defnD0}
D_0 := \{(u_1,u_2,0,0,\dots) \in \ell^2 \mid u_1^2 + u_2^2 \leq 1\} 
\end{equation}
is the map that was considered in Example \ref{exsigma1} and hence
\[
\sigma(S_0,df) = 1 \qquad \mbox{and} \qquad \tau(S_0,df) = 0.
\]
\end{example}

\begin{example}
\label{exg}
Our second model will be a proper Fredholm map of index one having as Pontryagin manifold a circle with vanishing $\sigma$. We start from the smooth map
\[
g_0: \mC^2 \rightarrow \mC\times \mR \cong \mR^3, \qquad (z_1,z_2) \mapsto (2 z_1 \bar{z}_2, |z_1|^2-|z_2|^2),
\]
which satisfies
\[
|g_0(z_1,z_2)|^2 = ( |z_1|^2 + |z_2|^2)^2 \qquad \forall (z_1,z_2)\in \mC^2,
\]
and hence is proper and maps $S^3$ to $S^2$. Its restriction to $S^3$ is the Hopf fibration. Its lift to $\ell^2$ is the proper Fredholm map of index one
\[
g: \ell^2 \rightarrow \ell^2, \qquad (u_1,u_2,u_3,\dots) \mapsto (g_0(u_1+ i u_2,u_3+iu_4),u_5,  \dots).
\]
The vector $e_3$ is a regular value of $g$ and its inverse image is the circle $S_0$ that we introduced in (\ref{defnS0}):
\[
g^{-1}(e_3) = S_0.
\]
Therefore, $(S_0,dg)$ is a one-dimensional framed submanifold of $\ell^2$. The circle $S_0$ can be seen as the boundary of the embedded disk
\begin{equation}
\label{halfsphere}
\varphi: \mD \rightarrow \ell^2, \qquad (x,y) \mapsto (x,y,\sqrt{1-x^2-y^2},0,0,\dots),
\end{equation}
and since $dg(u)$ is surjective for every $u\in \varphi(\mD)$ we have
\[
\sigma(S_0,dg) = \eta(dg\circ \varphi,\Phi_1^{\mathrm{sing}}(\ell^2)) = 0 \qquad \mbox{and} \qquad \tau(S_0,dg) = 1.
\]
\end{example}

\begin{example}
\label{exh}
The third example will be a proper Fredholm map of index one having as Pontryagin manifold a pair of circles with vanishing $\sigma$. Consider the map
\[
p_0: \mC\times \mR \rightarrow \mC \times \mR, \qquad (z,t) \mapsto (z,t^2),
\]
which maps $S^2$ into the paraboloid
\[
P:= \{(z,1-|z|^2) \mid z\in \mC\}.
\]
More precisely, $p_0$ maps both the upper and the lower hemispheres of $S^2$ diffeomorphically onto the open disk
\[
P^+:= \{(z,t)\in P \mid t>0\},
\]
and is the identity on the equator. By composing the Hopf map $g_0$ from Example \ref{exg} with $p_0$ we obtain the map
\[
h_0:= p_0\circ g_0: \mC^2 \rightarrow \mC \times \mR\cong \mR^3,
\]
sending $S^3$ into $P$. The point $(0,0,1)\in \mR^3$ is a regular value of $h_0$ and its inverse image consists of two circles:
\[
h_0^{-1}(0,0,1) = \{(z_1,0) \mid |z_1|=1\} \cup \{ (0,z_2) \mid |z_2|=1 \} \subset \mC^2.
\]
It is useful to construct an explicit cobordism from the union of these two circles to the empty set: Let $\chi: [0,1] \rightarrow \mR$ be a smooth monotonically increasing function such that $\chi=0$ on $[0,\frac{1}{4}]$, $\chi(\frac{1}{2})=1$, $\chi'(\frac{1}{2})>0$ and $\chi=2$ on $[\frac{3}{4},1]$. Let $\gamma: [0,1] \rightarrow P \subset \mR^3$ be the path
\[
\gamma(t):= \bigl( \chi(t), 0, 1 - \chi(t)^2 \bigr),
\]
and consider the smooth homotopy
\[
k_0: [0,1] \times \mC^2 \rightarrow \mC \times \mR \cong \mR^3, \qquad (t,z_1,z_2) \mapsto h_0(z_1,z_2) - \gamma(t).
\]
The map $k_0$ is easily shown to have the following properties:
\begin{enumerate}[(a$_0$)]
\item $dk_0(t,z_1,z_2)[(s,v)]  = dh_0(z_1,z_2)[v]$ for every $(t,z_1,z_2)\in ([0,\frac{1}{4}]\cup [\frac{3}{4},1]) \times \mC^2$ and every $(s,v)\in \mR \times \mC^2$;
\item $0$ is a regular value for $k_0$, the surface $k_0^{-1}(0)$ is contained in $[0,\frac{1}{2} ]\times \mC^2$ and its projection to the second factor is the annulus
\[
\Sigma_0 := \Bigl\{(z\cos s, z\sin s) \mid z\in \mC, \; |z|=1,\; s\in \bigl[0,{\textstyle \frac{\pi}{2}} \bigr] \Bigr\} \subset S^3 \subset \mC^2.
\]
\item $k_0^{-1}(0) \cap ([0,\frac{1}{4}] \times \mC^2) = [0,\frac{1}{4}] \times h_0^{-1}(0,0,1)$; in particular, the boundary of $k_0^{-1}(0)$ is given by the two circles $\{0\} \times h_0^{-1}(0,0,1)$.
\end{enumerate}
The surface $k_0^{-1}(0) \subset [0,1]\times \mC^2$ is the desired cobordism from $h_0^{-1}(0,0,1)$ to the empty set.

The maps $p_0$ and $h_0$ can be lifted to smooth self-maps of $\ell^2$ as follows:
\[
\begin{split}
p: \ell^2 \rightarrow \ell^2, \qquad (u_1,u_2,u_3, \dots) &\mapsto (p_0(u_1+iu_2,u_3), u_4, u_5,\dots), \\
h: \ell^2 \rightarrow \ell^2, \qquad (u_1,u_2,u_3, \dots) &\mapsto (h_0(u_1 + i u_2, u_3 + i u_4), u_5,u_6, \dots).
\end{split}
\]
The map $p$ is Fredholm of index zero, whereas $h$ is Fredholm of index one. The vector $e_3$ is a regular value of $h$ and
\[
h^{-1}(e_3) = S_0 \cup S_1,
\]
where $S_0$ is the circle introduced in (\ref{defnS0}) and $S_1$ is the circle
\begin{equation}
\label{defnS1}
S_1 := \{ (0,0,u_3,u_4, 0,0, \dots)\in \ell^2 | u_3^2+u_4^2 = 1\}.
\end{equation}
Therefore, 
\[
dh: \ell^2 \rightarrow \Phi_1(\ell^2) 
\]
is a framing of $S_0 \cup S_1$. The value of $\sigma(S_0,dh)$ and $\sigma(S_1,dh)$ are computed in Lemma \ref{sigmalemma} below. Similarly, the homotopy $k_0$ lifts to the Fredholm homotopy of index two
\[
k: [0,1] \times \ell^2 \rightarrow \ell^2, \qquad (t,u_1,u_2,u_3,\dots) \mapsto (k_0(t,u_1+iu_2,u_3 + i u_4),u_5,u_6, \dots).
\]
Conditions (a$_0$), (b$_0$) and (c$_0$) translate into:
\begin{enumerate}[(a)]
\item $dk(t,u) = S_L^* dh (u)$ for every $(t,u)\in ([0,\frac{1}{4}]\cup [\frac{3}{4},1]) \times \ell^2$;
\item $0$ is a regular value for $k$, the surface $k^{-1}(0)$ is contained in $[0,\frac{1}{2}]\times \ell^2$ and its projection to the second factor is the annulus
\begin{equation}
\label{Sigmadef}
\Sigma := \Bigl\{(x\cos s, y\cos s, x \sin s, y \sin s,0,0,\dots) \in \ell^2 \mid x^2+y^2=1,\; s\in \bigl[0,{\textstyle \frac{\pi}{2}} \bigr] \Bigr\}.
\end{equation}
\item $k^{-1}(0) \cap ([0,\frac{1}{4}] \times \ell^2) = [0,\frac{1}{4}] \times (S_0 \cup S_1)$; in particular, the boundary of $k^{-1}(0)$ is given by the two circles $\{0\} \times (S_0\cup S_1)$.
\end{enumerate}
Therefore, $(dk,k^{-1}(0))$ is a framed cobordism from $(S_0\cup S_1,dh)$ to $(\emptyset,dh)$.
\end{example}

We now compute the value of $\sigma$ at $(S_0,dh)$ and $(S_1,dh)$.

\begin{lemma}
\label{sigmalemma}
$\sigma(S_0,dh) = \sigma(S_1,dh) = 0$.
\end{lemma}

\begin{proof}
We can reduce the computation of $\sigma(S_0,dh)$ to the identity $\sigma(S_0,dg)=0$ established in Example \ref{exg}: The homotopy
\[
p_t := t p + (1-t) \, \mathrm{id} : \ell^2 \rightarrow \ell^2, \qquad t\in [0,1],
\]
is  Fredholm and satisfies
\[
dp_t(e_3) \in \mathrm{GL}(\ell^2) \qquad \forall t\in [0,1].
\]
Let $\varphi: \mD \rightarrow \ell^2$ be the embedding defined in (\ref{halfsphere}) and recall that $dg\circ \varphi$ does not meet the singular set $\Phi_1^{\mathrm{sing}}(\ell^2)$.
Since $g(S_0) = \{e_3\}$, we deduce that the homotopy
\[
dp_t \circ g \, dg \circ \varphi: \mD \rightarrow \Phi_1(\ell^2)
\]
maps $\partial \mD$ into $\Phi_1^0(\ell^2)$ for every $t\in [0,1]$. As this homotopy connects $dg\circ \varphi$ to $dh\circ \varphi$, we deduce that
\[
\sigma(dh,S_0) = \eta ( dh\circ \varphi, \Phi_1^{\mathrm{sing}}(\ell^2)) = \eta (dg\circ \varphi, \Phi_1^{\mathrm{sing}}(\ell^2)) = \sigma(S_0,dg) = 0.
\]
The computation of $\sigma(dh,S_1)$ can be reduced to the computation above. Indeed, the involution
\[
s_0: \mC^2 \rightarrow \mC^2, \qquad (z_1,z_2) \mapsto (\bar{z}_2,\bar{z}_1),
\]
satisfies $h_0\circ s_0=h_0$ and hence its lift
\[
s : \ell^2 \rightarrow \ell^2, \qquad (u_1,u_2,u_3, \dots) \mapsto (u_3,-u_4,u_1,-u_2,u_5,u_6, \dots),
\]
satisfies $h\circ s= h$. We have
\[
\sigma(dh,S_1) = \sigma(dh\circ s \, ds, S_0) = \sigma(dh,S_0) = 0,
\]
where the first equality follows from the fact that $s$ is a diffeomorphism mapping $S_0$ to $S_1$ and the second one by differentiating the identity $h \circ s = h$.
\end{proof}

We conclude this section by discussing the coherence of the orientations of the circles $S_0$ and $S_1$ with respect to the framing $dh$, as introduced at the end of Section \ref{prelsec}. The embedding
\[
\left[ 0 , {\textstyle \frac{\pi}{2}} \right] \times S^1 \rightarrow \ell^2, \qquad (s,z) \mapsto (\mathrm{Re}\, z \cos s, \mathrm{Im}\, z \cos s, \mathrm{Re}\, z \sin s, \mathrm{Im}\, z \sin s,0,0,\dots)
\]
has image $\Sigma$, see (\ref{Sigmadef}). We give $[0,\frac{\pi}{2}]$ and $S^1$ the standard orientations, $[0,\frac{\pi}{2} ]\times S^1$ the product orientation, $\Sigma$ the orientation that is induced by the above embedding and $S_0\cup S_1=\partial \Sigma$ the boundary orientation. This defines an orientation of these two circles that we shall refer to as the standard orientation of $S_0$ and $S_1$. One easily checks that the tangent vectors
\begin{equation}
\label{posor}
-e_2 \in T_{e_1} S_0 \qquad \mbox{and} \qquad e_4 \in T_{e_3} S_1
\end{equation}
are positively oriented.

\begin{lemma}
\label{orlemma}
The standard orientations of $S_0$ and $S_1$ are $dh$-coherent.
\end{lemma}

\begin{proof}
The path
\[
\alpha: \left[ 0, {\textstyle \frac{\pi}{2} }\right] \rightarrow \ell^2, \qquad t \mapsto (\cos t, 0, \sin t, 0 , 0, \dots)
\]
connects $e_1\in S_0$ to $e_3\in S_1$. The composition $F := dh\circ \alpha$ is the path of Fredholm operators index one
\[
F(t) \left( \begin{array}{c} v_1 \\ v_2 \\ v_3 \\ v_4 \\ \vdots \end{array} \right) =  \left( \begin{array}{c} 2 (v_1 \sin t + v_3 \cos t ) \\ 2 (v_2 \sin t - v_4 \cos t ) \\ 4 ( v_1 \cos t - v_3 \sin t) \cos 2t  \\ v_5 \\ \vdots \end{array} \right).
\]
The kernel of $F(t)$ has dimension one for every $t\in [0,\frac{\pi}{2}]\setminus \{\frac{\pi}{4}\}$ and dimension two for $t=\frac{\pi}{4}$. Therefore, $F(t)$ fails to be surjective for $t=\frac{\pi}{4}$.
In order to exhibit a non-vanishing section of the line bundle
\[
F^* \det \rightarrow  \left[ 0, {\textstyle\frac{\pi}{2} }\right],
\]
it is convenient to stabilize $F$: We choose a path of linear mappings
\[
G(t) : \mR^k \rightarrow \ell^2, \qquad t\in  \left[ 0, {\textstyle \frac{\pi}{2} }\right],
\]
such that the operators
\[
F_G(t): \mR^k \oplus \ell^2 \rightarrow \ell^2, \qquad (s,u) \mapsto G(t) s + F(t)u,
\]
are surjective. This surjectivity implies that the path
\[
t\mapsto \ker F_G(t)
\]
is continuous into the Grassmannian of $(k+1)$-planes in $\mR^k \oplus \ell^2$. Any continuous section of $\Lambda^{k+1}( \ker F_G)$ induces a continuous section of $F^*\det$ thanks to the canonical isomorphism
\[
\Lambda^{\max}(\ker F(t)) \otimes \Lambda^{\max}(\coker F(t))^* \cong \Lambda^{\max} ( \ker F_G(t))
\]
that is induced by the exact sequence
\[
0 \rightarrow \ker F(t) \rightarrow \ker F_G(t) \rightarrow \mR^k \rightarrow \coker F(t) \rightarrow 0,
\]
where the first map is the inclusion, the second one is the restriction of the projection $\mR^k \oplus \ell^2\rightarrow \mR^k$ and the third one is the composition of $G(t)$ by the quotient projection
\[
\ell^2 \rightarrow \frac{\ell^2}{\mathrm{im}\, F(t) } = \coker F(t).
\]
See \cite[Appendix]{fh93} for more details on this form of stabilization and its use in the definition of the vector bundle structure of the determinant bundle. In our case, a convenient stabilization is obtained by choosing $k=1$ and 
\[
G(t) : \mR \rightarrow \ell^2, \qquad G(t)s := 4 e_3 \sin 2t.
\]
The surjectivity of $F_G(t)$ follows from the fact that the matrix
\[
\left( \begin{array}{ccccc} 0 & \sin t & 0 & \cos t & 0 \\ 0 & 0 & \sin t & 0 & -\cos t \\ 2 \sin 2t & 2\cos t \cos 2t & 0 & - 2 \sin t \cos 2t& 0 \end{array} \right)
\]
has rank three for every $t\in [0,\frac{\pi}{2}]$. The vectors
\[
\begin{split}
u(t) & := - e_0 \cos 2t + (e_1 \cos t - e_3 \sin t) \sin 2t, \\ 
v(t) & := e_2 \cos t + e_4 \sin t,
\end{split}
\]
where $e_0 := (1,0) \in \mR\times \ell^2$, depend continuously on $t$ and form a basis of $\ker F_G(t)$, for every $t\in [0,\frac{\pi}{2}]$. Therefore, the path
\[
t\mapsto u(t)\wedge v(t)
\]
is a non-vanishing section of $\Lambda^2(\ker F_G)$. Its value for $t=0$ is
\[
u(0) \wedge v(0) = - e_0 \wedge e_2 = e_0 \wedge (-e_2),
\]
whereas its value for $t=\frac{\pi}{2}$ is
\[
u\left({\textstyle \frac{\pi}{2}}\right) \wedge v\left({\textstyle \frac{\pi}{2}}\right) = e_0 \wedge e_4.
\]
The above formulas imply that the orientations of $S_0$ and $S_1$ that are defined by declaring the vectors 
\[
-e_2 \in T_{e_1} S_0 \qquad \mbox{and} \qquad e_4 \in T_{e_3} S_1
\]
to be positively oriented are $dh$-coherent. But these are precisely the standard orientations of $S_0$ and $S_1$, see (\ref{posor}).
\end{proof}

\section{Reduction to the models}
\label{redsec}

In this section, we will show how one-dimensional framed submanifolds of $\ell^2$ satisfying suitable assumptions can be reduced to the three models that we introduced in the previous section. In the following lemma, $S_0$ denotes the circle
\[
S_0 = \{ (u_1,u_2,0,0,\dots) \in \ell^2 \mid u_1^2+ u_2^2 = 1\}
\]
that is framed by the maps $df: \ell^2 \rightarrow \Phi_1(\ell^2)$ and $dg: \ell^2 \rightarrow \Phi_1(\ell^2)$ introduced in Examples \ref{exf} and \ref{exg}. This circle is the boundary of the embedded disk
\[
D_0 = \{(u_1,u_2,0,0,\dots) \in \ell^2 \mid u_1^2+u_2^2\leq 1\},
\]
and we recall that
\[
\begin{split}
\eta(df|_{D_0}, \Phi_1^{\mathrm{sing}}(\ell^2)) &= \sigma(df,S_0) = 1, \\
\eta(dg|_{D_0}, \Phi_1^{\mathrm{sing}}(\ell^2)) &= \sigma(dg,S_0) = 0.
\end{split}
\]
We denote by $B_r$ the open ball of radius $r$ centered at the origin in $\ell^2$ and by $B_r^c:= \ell^2 \setminus B_r$ its complement. 

\begin{lemma}
\label{model1}
Let $A: \ell^2 \rightarrow \Phi_1(\ell^2)$ be a framing of $S_0$ in $\ell^2$. Then  there exists a homotopy
\[
H:[0,1] \times \ell^2 \rightarrow \Phi_1(\ell^2)
\]
such that:
\begin{enumerate}[(i)]
\item $H(0,u)=A(u)$ for every $u\in \ell^2$;
\item $H(t,u)=A(u)$ for every $(t,u)\in [0,1]\times B_2^c$;
\item $\ker H(t,u)=T_u S_0$ for every $(t,u)\in [0,1]\times S_0$;
\item if $\sigma(S_0,A)=1$ then $H(1,u) = df(u)$ for every $u\in D_0$, whereas if $\sigma(S_0,A)=0$ then $H(1,u) = dg(u)$ for every $u\in D_0$.
\end{enumerate}
\end{lemma}

\begin{proof}
We define 
\[
F: D_0 \rightarrow \Phi_1(\ell^2)
\]
to be the restriction of  either $df$, if $\sigma(S_0,A)=1$, or $dg$, if $\sigma(S_0,A)=0$. Therefore,
\begin{equation}
\label{therightone}
\eta(F|_{D_0},\Phi_1^{\mathrm{sing}}(\ell^2)) = \eta(A|_{D_0},\Phi_1^{\mathrm{sing}}(\ell^2)).
\end{equation}
The loops $A|_{S_0}$ and $F|_{S_0}$ are contractible in $\Phi_1(\ell^2)$ and, since the homomorphism $\partial_1$ in the exact sequence (\ref{exact}) vanishes, also in $\Phi_1^0(\ell^2)$. Therefore, there exists a homotopy
\[
H_0:[0,1]\times S_0 \rightarrow \Phi_1^0(\ell^2)
\]
from $A|_{S_0}$ to $F|_{S_0}$. Denote by $\mathrm{O}(\ell^2)$ the orthogonal group of $\ell^2$ and choose a map
\[
U: [0,1] \times S_0 \rightarrow \mathrm{O}(\ell^2)
\]
such that
\[
\begin{split}
U(0,u) &= I \qquad \forall u\in S_0, \\ U(t,u) T_u S_0 &= \ker H_0(t,u) \qquad \forall (t,u)\in [0,1]\times S_0. 
\end{split}
\]
Since $\ker H_0(1,u)=\ker F(u) = T_u S_0$ for every $u\in S_0$ and $U$ is orthogonal, for every $u\in S_0$ the operator $U(1,u)$ maps the orthogonal complement of $T_u S_0$ onto itself. Since this orthogonal complement is infinite dimensional, the fact that the orthogonal group of an infinite dimensional Hilbert space is contractible implies that we can further assume that $U(1,u)$ is the identity on the orthogonal complement of $T_u S_0$, for every $u\in S_0$. The homotopy
\[
H_1:[0,1]\times S_0 \rightarrow \Phi_1(\ell^2), \qquad H_1(t,u) := H_0(t,u) U(t,u),
\]
satisfies
 \[
\begin{split}
H_1(0,u) &= A(u)  \qquad \forall u\in S_0, \\ \ker H_1(t,u) &= T_u S_0\qquad \forall (t,u)\in [0,1]\times S_0, \\  H_1(1,u) &= F(u) \qquad \forall u\in S_0.
\end{split}
\]
Indeed, the last identity follows from the fact that the operators on both sides of the equality have the same kernel $T_u S_0$ and agree on its orthogonal complement.
Extend the map $H_1$ to the set
\begin{equation}
\label{largeset}
(\{0\} \times \ell^2) \cup ([0,1]\times B_2^c) \cup ([0,1]\times S_0)
\end{equation}
by setting $H_1(t,u)=A(u)$ on the union of the first two sets. Since the set (\ref{largeset}) is a retract of $[0,1]\times \ell^2$, we can extend $H_1$ to a homotopy
\[
H_1 : [0,1] \times \ell^2 \rightarrow \Phi_1(\ell^2).
\]
This homotopy satisfies (i), (ii) and (iii), whereas the identity appearing in (iv) is for now guaranteed only on $S_0$. We will achieve the identity on the whole $D_0$ by a further homotopy.

Set $A'(u):= H_1(1,u)$. Since $A|_{D_0}$ and $A'|_{D_0}$ are homotopic through a homotopy mapping $S_0$ into $\Phi_1^0(\ell^2)$, we have
\[
\eta(A'|_{D_0}, \Phi_1^{\mathrm{sing}}(\ell^2)) = \eta(A|_{D_0}, \Phi_1^{\mathrm{sing}}(\ell^2)).
\]
Using also the fact that $A'$ and $F$ coincide on $S_0$ and the identity (\ref{therightone}), Lemma \ref{intnumb} together with Remark \ref{hom-rel-bdry} implies the existence of a homotopy
\[
H_2 : [0,1] \times D_0 \rightarrow \Phi_1(\ell^2)
\]
such that
\[
\begin{split}
H_2(0,u) &= A'(u) \qquad \forall u\in D_0, \\ H_2(1,u) &= F(u) \qquad \forall u\in D_0, \\ H_2(t,u) = A'(u) &= F(u) \qquad \forall (t,u)\in [0,1] \times S_0.
\end{split}
\]
We now extend $H_2$ to a homotopy
\[
H_2 : [0,1] \times \ell^2 \rightarrow \Phi_1(\ell^2)
\]
satisfying
\[
\begin{split}
H_2(0,u)&= A'(u) \qquad \forall u\in \ell^2, \\
H_2(t,u) &= A'(u) \qquad \forall (t,u)\in [0,1]\times B_2^c,
\end{split}
\]
using that fact that
\[
(\{0\} \times \ell^2) \cup ([0,1]\times B_2^c) \times ([0,1]\times D_0)
\]
is a retract of $[0,1]\times \ell^2$. The concatenation of the homotopies $H_1$ and $H_2$ gives us a homotopy $H$ with the desired properties.
\end{proof}

In the next lemma, we consider certain framings of $S_0 \cup S_1$, where as before
\[
S_0 = \{ (u_1,u_2,0,0,\dots) \in \ell^2 \mid u_1^2+ u_2^2 = 1\}, \qquad S_1 = \{ (0,0,u_3,u_4,0,\dots)\in \ell^2 \mid  u_3^2+u_4^2=1\}.
\]
Recall that that $S_0 \cup S_1$ is framed by the map $dh: \ell^2 \rightarrow \Phi_1(\ell^2)$ from Example \ref{exh} and
\begin{equation}
\label{repeat}
\sigma(S_0,dh) = \sigma(S_1,dh) = 0,
\end{equation}
by Lemma \ref{sigmalemma}, and that the standard orientations of $S_0$ and $S_1$ are $dh$-coherent, see Lemma \ref{orlemma}.
Moreover, the union of these two circles is the boundary of the annulus $\Sigma$ that is defined in (\ref{Sigmadef}).
\begin{lemma}
\label{model2}
Let $A: \ell^2 \rightarrow \Phi_1(\ell^2)$ be a framing of $S_0 \cup S_1$ such that
\begin{equation}
\label{ass}
\sigma(A,S_0) = \sigma(A,S_1) = 0,
\end{equation}
and such that the standard orientations of $S_0$ and $S_1$ are $A$-coherent. Then there exists a homotopy
\[
H: [0,1]\times \ell^2 \rightarrow \Phi_1(\ell^2)
\]
such that:
\begin{enumerate}[(i)]
\item $H(0,u) = A(u)$ for every $u\in \ell^2$;
\item $H(t,u) = A(u)$ for every $(t,u) \in [0,1] \times  B^c_2$;
\item $\ker H(t,u) = T_u (S_0 \cup S_1)$ for every $(t,u) \in [0,1] \times (S_0 \cup S_1)$;
\item $H(1,u) = dh(u)$ for every $u\in \Sigma$.
\end{enumerate}
\end{lemma}

\begin{proof} 
By gluing two disks along the two boundary components of $\Sigma$, we obtain a two-sphere in $\ell^2$. The intersection numbers with $\Phi_1^{\mathrm{sing}}(\ell^2)$ of the restrictions to this sphere of both $A$ and $dh$ vanish because the sphere is contractible in $\ell^2$. Then the identity (\ref{repeat}) and the assumption (\ref{ass}) imply that
\[
\eta(dh|_{\Sigma},   \Phi_1^{\mathrm{sing}}(\ell^2)) = \eta(A|_{\Sigma}, \Phi_1^{\mathrm{sing}}(\ell^2)) = 0.
\]
By Lemma \ref{annulus} there exists a homotopy
\[
H_0: [0,1]\times \Sigma  \rightarrow \Phi_1(\ell^2)
\]
such that:
\begin{equation}
\label{H0}
\begin{split}
H_0(0,u) &= A(u) \qquad \forall u\in \Sigma, \\
H_0(1,u) &= dh(u) \qquad \forall u\in \Sigma, \\
H_0(t,u) & \in \Phi_1^0(\ell^2) \qquad \forall (t,u)\in [0,1] \times \partial \Sigma.
\end{split}
\end{equation}
The line bundle
\begin{equation}
\label{LB}
H_0^* \det \rightarrow [0,1]\times \Sigma
\end{equation}
is trivial because the inclusion of $\{0\}\times S_0$ into its basis is a homotopy equivalence and the restriction of (\ref{LB}) to $\{0\} \times S_0$ is the trivial line bundle
\[
\Lambda^1(\ker A|_{S_0}) = \Lambda^1(TS_0) \cong TS_0.
\]
Let $\xi$ be a non-vanishing section of (\ref{LB}) whose restriction to $\{0\} \times S_0$ gives us the standard orientation of $S_0$. Since we are assuming that the standard orientations of $S_0$ and $S_1$ are $A$-coherent, the restriction of $\xi$ to $\{0\} \times S_1$ gives us the standard orientation of $S_1$. Consider the orientations of $S_0$ and $S_1$ that are induced by the restrictions of $\xi$ to $\{1\} \times S_0$ and $\{1\}\times S_1$. Being $dh$-coherent, these orientations either both agree with the standard orientations or both disagree with them, thanks to Lemma \ref{orlemma}. The first case corresponds to the situation in which the line bundles
\[
\begin{split}
\ker H_0|_{[0,1]\times S_0} & \rightarrow [0,1]\times S_0, \\
\ker H_0|_{[0,1]\times S_1} & \rightarrow [0,1]\times S_1,
\end{split}
\]
which satisfy
\[
\begin{split}
\ker H_0|_{\{0\} \times S_0} &= TS_0 = \ker H_0|_{\{1\} \times S_0}, \\
\ker H_0|_{\{0\} \times S_1} &= TS_1 = \ker H_0|_{\{1\} \times S_1},
\end{split}
\]
induce trivial line bundles over $\mT \times S_0$ and $\mT \times S_1$, where $\mT$ denotes the quotient $[0,1]/\{0,1\}\cong S^1$. In this case, we set $\epsilon:= 1$. The second situation corresponds to the case in which the kernel of $H_0$ induces non-trivial line bundles over both $\mT \times S_0$ and $\mT \times S_1$. In this case, we set $\epsilon:= -1$. The relevant fact here is that Lemma \ref{orlemma} and the second assumption on $A$ exclude the possibility that one of these line bundles is trivial and the other one is not.

Then we can find a continuous family of isometries
\[
U(t,u) : T_u \partial \Sigma \rightarrow \ker H_0(t,u), \qquad (t,u)\in [0,1]\times \partial \Sigma,
\]
such that
\[
U(0,u) = \mathrm{id}_{T_u \partial \Sigma}, \qquad U(1,u) = \epsilon \, \mathrm{id}_{T_u \partial \Sigma} \qquad \forall u\in \partial \Sigma.
\]
We can extend these isometries to a map
\[
U : [0,1] \times \partial \Sigma \rightarrow \mathrm{O}(\ell^2)
\]
such that
\begin{equation}
\label{U1}
\begin{split}
U(t,u) T_u \partial \Sigma &= \ker H_0(t,u), \\
U(0,u) = I, \qquad U(1,u) &= \epsilon I \qquad \forall u\in \partial \Sigma.
\end{split}
\end{equation}
Using the fact that $\Sigma$ is contained in the closed unit ball and Kuiper's theorem stating that $\mathrm{O}(\ell^2)$ is contractible, we can extend $U$ to a map
\[
U : [0,1] \times \ell^2 \rightarrow \mathrm{O}(\ell^2)
\]
such that
\begin{equation}
\label{U2}
\begin{split}
U(0,u) &= I \qquad \forall u\in \ell^2, \\
U(1,u) &= \epsilon I\qquad  \forall u\in \Sigma, \\
U(t,u) &= I \qquad \forall (t,u)\in [0,1]\times B_2^c.
\end{split}
\end{equation}
By (\ref{H0}), (\ref{U1}) and (\ref{U2}), the homotopy
\[
H: [0,1]\times \ell^2 \rightarrow \Phi_1(\ell^2), \qquad H(t,u) := U(t,u) H_0(t,u) U(t,u),
\]
satisfies the required properties.
\end{proof}

\section{How to eliminate or match certain circles}
\label{blockssec}

In the next proposition we do not assume the framing $A$ to be spin (and the Hilbert manifold $M$ need not be simply connected).

\begin{proposition}
\label{kill1}
Let $(X,A)$ be a one-dimensional framed submanifold of the Hilbert manifold $M$. Assume that $\varphi: \mD \rightarrow M$ maps $\partial \mD$ homeomorphically onto a component $X_0$ of $X$ and
\[
\eta(A\circ \varphi, \Phi_1^{\mathrm{sing}}(\mH)) = 1.
\]
Then $(X,A)$ is framed cobordant to $(X\setminus X_0,A')$ for a suitable framing $A': M \rightarrow \Phi_1(\mH)$ of $X\setminus X_0$.
\end{proposition}

\begin{proof}
Up to a small perturbation, we may assume that $\varphi$ is a smooth embedding. We identify some open subset  $M_0$ of $M$ 
with the model Hilbert space $\mH$, and the latter with $\ell^2$. Since the group of diffeomorphisms of $M$ that are isotopic to the identity acts transitively on embedded disks, thanks to Lemma \ref{functors} we may assume that the image of the embedding $\varphi$ is the disk $D_0\subset \ell^2$ that is defined in (\ref{defnD0}), and hence that $X_0$ is the circle $S_0=\partial D_0$. We can also assume that $M_0\cap (X\setminus S_0)=\emptyset$. Since the identification 
\[
\mD \rightarrow D_0, \qquad (x,y) \mapsto (x,y,0,0,\dots)
\]
is isotopic to either $\varphi$ or $z\mapsto \varphi(\bar{z})$ and the composition of $A$ with either of the latter two embeddings has intersection number one with $\Phi_1^{\mathrm{sing}}(\mH)$, we have
\[
\eta(A|_{D_0}, \Phi_1^{\mathrm{sing}}(\mH))=1.
\]
We now apply Lemma \ref{model1} to the restriction of $A$ to $M_0\cong \ell^2$, where
\[
\sigma(S_0,A|_{\ell^2}) = \eta(A|_{D_0}, \Phi_1^{\mathrm{sing}}(\mH))=1.
\]
The homotopy given by this lemma extends to the whole $M$ and defines a framing of the trivial cobordism $[0,1]\times X$ starting at $A$ and ending at some framing $\tilde{A}$ of $X$ such that 
\[
\tilde{A}|_{D_0} = df|_{D_0}.
\]
It suffices to show that $(X,\tilde{A})$ is framed cobordant to $(X\setminus S_0,A')$ for some framing $A'$ of $X\setminus S_0$.

Consider the smooth Fredholm map of index two
\[
q: [0,1] \times \ell^2 \rightarrow \ell^2, \quad (t,u) \mapsto f(u) - \chi(t) e_1 = \bigl( u_1^2+u_2^2 - \chi(t), u_3, u_4, \dots \bigr),
\]
where $\chi:[0,1]\rightarrow \mR$ is a smooth monotonically decreasing function such that
\[
\chi(t) = 1 \quad \forall t\in \left[ 0 , {\textstyle \frac{1}{4}} \right], \qquad \chi\left( {\textstyle \frac{1}{2}} \right) = 0, \qquad \chi' \left({\textstyle \frac{1}{2}} \right) <0.
\]
One readily checks that $0$ is a regular value of $q$ and hence the set
\[
W_0:= q^{-1}(0)
\]
is a two-dimensional submanifold of $[0,1]\times \ell^2$ (it is diffeomorphic to a disk). We have
\[
W_0\cap \left( \bigl[ 0 , {\textstyle \frac{1}{4}} \bigr] \times \ell^2\right) = \bigl[ 0, {\textstyle \frac{1}{4}} \bigr] \times S_0,
\]
and 
\[
W_0 \cap\left(   \bigl( {\textstyle \frac{1}{2}} , 1 \bigr] \times \ell^2 \right) = \emptyset,
\]
because $\chi(t)<0$ for every $t>\frac{1}{2}$. Therefore, the submanifold 
\[
W:= \bigl( [0,1]\times (X\setminus S_0) \bigr) \cup W_0
\]
is a cobordism from $X$ to $X\setminus S_0$. There remains to construct a framing of this cobordism starting at $\tilde{A}$. Since $0$ is a regular value of $f$ and $W_0=q^{-1}(0)$, we have
\[
\ker dq(t,u) = T_{(t,u)} W_0 \qquad \forall (t,u)\in W_0.
\]
For $t\in [0,\frac{1}{4}]$ we have
\[
dq(t,u) [s,v] = ( 2(u_1 v_1 + u_2 v_2), v_3, v_4,\dots) \qquad \forall (s,v)\in \mR \times \ell^2,
\]
and hence 
\[
dq(t,u) = S_L^* \tilde{A}(u) \qquad \forall (t,u) \in \left[ 0 , {\textstyle \frac{1}{4} } \right] \times D_0.
\]
Therefore, the map 
\[
B: \left(  \bigl[ 0 , {\textstyle \frac{1}{4} }\bigr] \times M \right) \cup \left( \bigl[0,{\textstyle \frac{2}{3}}\bigr] \times (M\setminus M_0) \right) \cup  \left(  \bigl[ 0 , {\textstyle \frac{2}{3} } \bigr] \times D_0 \right) \rightarrow \Phi_2(\mR\oplus \mH,\mH)
\]
given by 
\[
B(t,u) := \left\{ \begin{array}{ll} S^*_L \tilde{A}(u) & \mbox{on }  \left(  \bigl[ 0 , \frac{1}{4} \bigr] \times M \right) \cup \left( \bigl[0,\frac{2}{3}\bigr] \times (M\setminus M_0) \right), \\ dq(t,u) & \mbox{on }  \bigl[ 0 , \frac{2}{3} \bigr] \times D_0, \end{array} \right.
\]
is well defined and continuous. Since its domain is a retract of $[0,\frac{2}{3}]\times M$, it extends to a map
\[
B: \bigl[ 0 , {\textstyle \frac{2}{3}} \bigr] \times M \rightarrow \Phi_2(\mR\times \mH,\mH).
\]
We now set
\[
A'(x) := S^*_R B \bigl( {\textstyle \frac{2}{3}} , x \bigr) \qquad \forall x\in M,
\]
and notice that
\[
B \bigl( {\textstyle \frac{2}{3}},x \bigr) = S_L^* A'(x) \qquad \forall x\in M\setminus M_0,
\]
and in particular on $X\setminus S$. We now extend $B$ to $[0,1]\times M$ by setting
\[
B(t,x) = \vartheta(t) B \bigl( {\textstyle \frac{2}{3}}, x\bigr) + (1-\vartheta(t)) S_L^* A'(x),
\]
where $\vartheta: [\frac{2}{3},1] \rightarrow [0,1]$ is a continuous function such that $\vartheta(\frac{2}{3})=1$ and $\vartheta(t)=0$ for every $t\in [\frac{3}{4},1]$. The map $B$ is easily seen to be a framing of the cobordism $W$ starting at $\tilde{A}$. We conclude that $(X,A)$ and $(X\setminus S, A')$ are framed cobordant.
\end{proof}

The second result of this section allows us to eliminate certain pairs of circles.

\begin{proposition}
\label{kill2}
Let $(X,A)$ be a one-dimensional framed submanifold of the simply connected Hilbert manifold $M$ with $A$ spin. Let $X_0\cong S^1$ and $X_1\cong S^1$ be distinct components of $X$ with
\[
\sigma(X_0,A) = \sigma(X_1,A) = 0.
\]
Then $(X,A)$ is framed cobordant to $(X\setminus (X_0 \cup X_1), A')$ for a suitable framing $A': M \rightarrow \Phi_1(\mH)$ of $X\setminus (X_0\cup X_1)$.
\end{proposition}

\begin{proof}
Since $M$ is simply connected, we can endow the circles $X_0$ and $X_1$ with $A$-coherent orientations. We identify some open subset $M_0$ of $M$ with the model Hilbert space $\mH$ and the latter with $\ell^2$. Therefore, $M_0$ contains the circles $S_0$, $S_1$ and the annulus $\Sigma$ that are introduced in Section \ref{modelsec}.  Since $M$ is simply connected and has dimension greater than three, the identity component of its diffeomorphism group acts transitively on equipotent finite collections of oriented embedded circles. Thanks to Lemma \ref{functors}, we can assume that $X_0=S_0$, $X_1=S_1$ and that the standard orientations of these circles are $A$-coherent. By Lemma \ref{model2}, we can further assume that $A=dh$ on $\Sigma$.

Let 
\[
k: [0,1]\times \ell^2 \rightarrow \ell^2
\]
be the smooth homotopy introduced in Example \ref{exh} and set
\[
W_0 := k^{-1}(0).
\]
Then
\[
W:=W_0 \cup \bigl( [0,1]\times (X \setminus (S_0 \cup S_1)) \bigr)
\]
is a cobordism from $X$ to $X \setminus (S_0 \cup S_1)$. By Property (a) in Example \ref{exh} and by the identity $A|_{\Sigma} = dh|_{\Sigma}$, the map
\[
B: \Bigl(  \bigl[ 0 , {\textstyle \frac{1}{4}} \bigr] \times M \Bigr) \cup \Bigl( \bigl[0, {\textstyle \frac{2}{3}} \bigr] \times (M\setminus M_0) \Bigr) \cup  \Bigl(  \bigl[ 0 , {\textstyle \frac{2}{3} } \bigr] \times \Sigma \Bigr) \rightarrow \Phi_2(\mR\oplus\mH,\mH)
\]
given by
\[
B(t,u) := \left\{ \begin{array}{ll} S^*_L A(u) & \mbox{on }  \Bigl(  \bigl[ 0 , \frac{1}{4} \bigr] \times M \Bigr) \cup \Bigl( \bigl[0,\frac{2}{3}\bigr] \times (M\setminus M_0) \Bigr), \\ dk(t,u) & \mbox{on }  \bigl[ 0 , \frac{2}{3} \bigr] \times D, \end{array} \right.
\]
is well defined and continuous. Since its domain is a retract of $[0,\frac{2}{3}]\times M$, it extends to a map
\[
B: \bigl[ 0 , {\textstyle \frac{2}{3} }\bigr] \times M \rightarrow \Phi_2(\mR\oplus \mH,\mH).
\]
We now set
\[
A'(x) := S^*_R B \bigl( {\textstyle \frac{2}{3} }, x \bigr) \qquad \forall x\in M,
\]
and notice that
\[
B \bigl( {\textstyle \frac{2}{3}},x \bigr) = S_L^* A'(x) \qquad \forall x\in M\setminus M_0,
\]
and in particular on $X\setminus (S_0\cup S_1)$. We now extend $B$ to $[0,1]\times M$ by setting
\[
B(t,x) = \vartheta(t) B \bigl( {\textstyle \frac{2}{3}}, x\bigr) + (1-\vartheta(t)) S_L^* A'(x),
\]
where $\vartheta: [\frac{2}{3},1] \rightarrow [0,1]$ is a continuous function such that $\vartheta(\frac{2}{3})=1$ and $\vartheta(t)=0$ for every $t\in [\frac{3}{4},1]$. Thanks to the properties (a), (b) and (c) from Example \ref{exh}, the map $B$ is a framing of the cobordism $W$ starting at $A$. We conclude that $(X,A)$ and $(X\setminus (S_0\cup S_1), A')$ are framed cobordant.
\end{proof}

We conclude this section by the following result, telling us how to construct a framed cobordism matching two circles.

\begin{proposition}
\label{match2}
Let $X_0$ and $X_1$ be embedded circles in the simply connected Hilbert manifold $M$. Let $A_0,A_1: M \rightarrow \Phi_1(\mH)$ be homotopic and spin framings of $X_0$ and $X_1$, respectively, with
\begin{equation}
\label{theass}
\sigma(X_0,A_0) = \sigma(X_1,A_1).
\end{equation}
Then $(X_0,A_0)$ and $(X_1,A_1)$ are framed cobordant. 
\end{proposition}

\begin{proof}
We identify some open subset of $M$ with the model Hilbert space $\mH$ and the latter with $\ell^2$. In particular, we see the disk $D_0$ and its boundary $S_0$ as subsets of $M$.
Thanks to Lemma \ref{model1}, we may assume that $X_0=X_1=S_0$ and that
\[
A_0|_{D_0} = A_1|_{D_0}.
\]
Indeed, Lemma \ref{model1} gives us the same model for $A_0|_{D_0}$ and $A_1|_{D_0}$ thanks to the assumption (\ref{theass}).  Now let 
\[
H_0 : [0,1] \times M \rightarrow \Phi_1(\mH)
\]
be a homotopy from $A_0$ to $A_1$. Since $A_0$ and $A_1$ coincide on $D_0$, the paths
\[
t\mapsto H_0(t,u),\qquad u\in D_0,
\]
are closed in $\Phi_1(\mH)$ and hence define an element $a$ of $\pi_1(\Phi_1(\mH))=\mZ_2$ that is independent of the choice of $u\in D_0$ because $D_0$ is connected. Up to replacing $H_0$ by another homotopy from $A_0$ to $A_1$, we may assume that $a$ is the trivial element: Indeed, if this is not the case we can replace $H_0$ by the homotopy
\[
[0,1] \times M \rightarrow \Phi_1(\mH), \qquad (t,x) \mapsto H_0(t,x) T(t),
\]
where $T:[0,1]\rightarrow \Phi_0(\mH)$ is a path such that $T(0)=T(1)=I$ and inducing the non-trivial element in $\pi_1(\Phi_0(\mH))=\mZ_2$. 

It is now easy to modify $H_0$ into a new homotopy $H$ having the further property
\begin{equation}
\label{extra}
H(t,u) = A_0(u) = A_1(u) \qquad \forall (t,u)\in [0,1] \times D_0.
\end{equation}
This homotopy induces a framing of the trivial cobordism $[0,1]\times S_0$, see Remark \ref{trivial2}, and shows that $(X_0,A_0)$ and $(X_1,A_1)$ are framed cobordant.

The construction of $H$ is standard, but for the sake of completeness we give the details. We denote by $Q$ the square $[0,1]^2$ and we define
\[
H_1: \partial Q \times D_0 \rightarrow \Phi_1(\mH)
\]
by
\[
H_1(s,t,u) := \left\{ \begin{array}{ll} H_0(t,u) & \mbox{if } s=0, \\ A_0(u)=A_1(u) & \mbox{if } t\in \{0,1\} \mbox{ or } s=1. \end{array} \right.
\]
Our assumption on $H_0$ implies that the closed curve $H_1|_{ \partial Q \times \{u\}}$ is contractible in $\Phi_1(\mH)$, for every $u\in D_0$. Therefore, we can extend $H_1$ to a map
\[
H_1: Q \times D_0 \rightarrow \Phi_1(\mH).
\]
Now we set
\[
C:=([0,1] \times \{0,1\} ) \cup (\{0\} \times [0,1]) \subset Q,
\]
and we define a map
\[
K: (C \times M) \cup (Q \times D_0) \rightarrow \Phi_1(\mH)
\]
by
\[
K(s,t,u) :=  \left\{ \begin{array}{ll} H_0(t,u) & \mbox{if } s=0, \\ A_0(u) & \mbox{if } t=0, \\ A_1(u) & \mbox{if } t=1, \\ H_1(s,t,u) & \mbox{if } u\in D_0. \end{array} \right.
\]
Since $C$ is a deformation retract of $Q$ and $D_0$ is a retract of some neighborhood of it in $M$, the set $(C \times M) \cup (Q \times D_0)$ is a retract of $Q\times M$ and the map $K$ extends to a map on the whole $Q\times M$. The restriction
\[
H : [0,1] \times M \rightarrow \Phi_1(\mH), \qquad (t,u) \mapsto K(1,t,u),
\]
is a homotopy from $A_0$ to $A_1$ satisfying (\ref{extra}).
\end{proof}

\section{Proof of the main theorem}
\label{proofsec}

We can now state and prove the main theorem of this article.

\begin{theorem}
Let $M$ be a simply connected Hilbert manifold. Then the one-dimensional framed submanifolds $(X_0,A_0)$ and $(X_1,A_1)$ are framed cobordant if and only if the following conditions are satisfied.
\begin{enumerate}[(i)]
\item The maps $A_0, A_1: M \rightarrow \Phi_1(\mH)$ are homotopic.
\item If $A_0$ and $A_1$ are both spin, then $\tau(X_0,A_0)=\tau(X_1,A_1)$. If $A_0$ and $A_1$ are not spin, then no further condition is necessary. 
\end{enumerate}
Furthermore, the map
\begin{equation}
\label{themap}
(X,A) \mapsto \left\{ \begin{array}{ll}  ([A], \tau(X,A)) & \mbox{if  $A$ is spin}, \\ \; [A] & \mbox{if $A$  is not spin}, \end{array} \right.
\end{equation}
induces a bijection
\[
\Omega_1^{\mathrm{fr}}(M) \cong ( [M,\Phi_1(\mH)]_{\mathrm{sp}} \times \mZ_2) \sqcup [M,\Phi(\mH)]_{\mathrm{ns}}.
\]
\end{theorem}

\begin{proof}
Let us first assume that $(X_0,A_0)$ and $(X_1,A_1)$ are framed cobordant. Then $A_0$ and $A_1$ are homotopic, as observed in Remark \ref{trivial1}, so (i) holds. In particular, they are either both spin or both not spin. If additionally $A_0$ and $A_1$ are spin, then Proposition \ref{tauprop} implies that $\tau(X_0,A_0) = \tau(X_1,A_1)$. Therefore, (ii) holds.

Conversely, assume that (i) and (ii) hold. We first consider the case in which $A_0$ and $A_1$ are not spin. This means that our only assumption is that $A_0$ and $A_1$ are homotopic.
Let $S$ be a connected component of $X_0$. Since $A_0$ is not spin, we can find a map
\[
\varphi: \mD \rightarrow M
\]
mapping $\partial \mD$ homeomorphically onto $S$ such that
\[
\eta(A\circ \varphi,\Phi_1^{\mathrm{sing}}(\mH)) = 1.
\]
By Proposition \ref{kill1}, $(X_0,A_0)$ is framed cobordant to $(X_0\setminus S,A_0')$, for some suitable framing $A'_0: M \rightarrow \Phi_1(\mH)$ of $X_0\setminus S$. By iterating this argument, we can eliminate all components of $X_0$ and obtain that $(X_0,A_0)$ is framed cobordant to $(\emptyset,\tilde{A}_0)$, for a suitable map $\tilde{A}_0: M \rightarrow \Phi_1(\mH)$ which is homotopic to $A_0$, and hence $(X_0,A_0)$ is framed cobordant to $(\emptyset, A_0)$, because any homotopy of Fredholm maps defines a cobordisms from the empty set to itself. By the same reason, also $(X_1,A_1)$ is framed cobordant to $(\emptyset,A_1)$. Since $A_0$ and $A_1$ are homotopic, the pairs $(\emptyset, A_0)$ and $(\emptyset,A_1)$ are framed cobordant and so are the given framed submanifolds $(X_0,A_0)$ and $(X_1,A_1)$.

Now we consider the case in which $A_0$ and $A_1$ are spin. Arguing as in the previous case, we can eliminate from $X_0$ all the components $S$ with $\sigma(S,A_0)=1$. Up to a first framed cobordism, we can then assume that all the components of $S$ of $X_0$ satisfy $\sigma(S,A_0)=0$. Thanks to Proposition \ref{kill2}, up to a second framed cobordism we can eliminate one by one pairs of components of $X_0$ until we remain with just one or no components. This shows that $(X_0,A_0)$ is framed cobordant to some $(X_0',A_0')$, where $X_0'$ is either empty or is consists of one component and satisfies $\sigma(X_0',A_0')=0$. The first case arises when $\tau(X_0',A_0') = \tau(X_0,A_0)=0$, the second one when $\tau(X_0',A_0') = \tau(X_0,A_0)=1$.

The same is true for $(X_1,A_1)$: $(X_1,A_1)$ is framed cobordant to some $(X_1',A_1')$, where $X_1'$ is either empty  - if $\tau(X_1',A_1') = \tau(X_1,A_1)=0$ - or consists of one component and satisfies $\sigma(X_1',A_1')=0$ -  if $\tau(X_1',A_1') = \tau(X_1,A_1)=0$. The assumption $\tau(X_0,A_0)=\tau(X_1,A_1)$ tells us that either $X_0'$ and $X_1'$ are both empty, or they both consist of one component. In the first case, a framed cobordism from $(X_0',A_0')$ to $(X_1',A_1')$ is induced by a homotopy from $A_0'$ to $A_1'$. In the second case, $(X_0',A_0')$ and $(X_1',A_1')$ are framed cobordant thanks to Proposition \ref{match2}. In either case, we conclude that $(X_0,A_0)$ is framed cobordant to $(X_1,A_1)$.

By what we have proven so far, the map (\ref{themap}) is well defined and injective. Let us check that it is surjective. If $\alpha\in [M,\Phi_1(\mH)]_{\mathrm{sp}}$, then any $A\in \alpha$ produces the trivial framed submanifold $(\emptyset,A)$ that satisfies $\tau(\emptyset,A)=0$. Its framed cobordism class is then mapped to $(\alpha,0)$ by (\ref{themap}). If $\alpha\in [M,\Phi_1(\mH)]_{\mathrm{ns}}$, then any $A\in \alpha$ produces the trivial framed submanifold $(\emptyset,A)$, whose framed cobordism class is mapped to $\alpha$ by (\ref{themap}). Therefore, it is enough to show that for any $\alpha\in [M,\Phi_1(\mH)]_{\mathrm{sp}}$ there exists a framed submanifold $(X,A)$ with $A\in \alpha$ and $\tau(X,A)=1$.

Fix some open subset $M_0$ of $M$ that can be identified with the model Hilbert space $\mH\cong \ell^2$. 
Since $\Phi_1(\mH)$ is connected, we can find a map $A_0 : M \rightarrow \Phi_1(\mH)$ in the homotopy class $\alpha$ such that for every $u\in M_0\cong \ell^2$ the operator $A_0(u)$ coincides with the operator
\[
T: \ell^2 \rightarrow \ell^2, \qquad (v_1,v_2,v_3, v_4, v_5, \dots) \mapsto (0,0,0,v_5,v_6,\dots).
\]
We can modify $A_0$ and obtain a new map $A: M \rightarrow \Phi_1(\mH)$ that coincides with $A_0$ outside of $M_0$ and
\[
A(u) = \chi(u) dg(u) + (1-\chi(u)) T  
 \qquad \forall u\in M_0 \cong \ell^2,
\]
where $g$ is the Fredholm map from Example \ref{exg} and
$\chi: \ell^2 \rightarrow \mR$ is a continuous function such that $\chi=1$ on $B_2$ and $\chi=0$ on $B_3^c$. Note that $A(u)$ is a Fredholm operator for every $u\in \ell^2$ because the difference $dg(u)-T$ has finite rank. Moreover, the map $A$ is homotopic to $A_0$ by the homotopy
\[
(t,u) \mapsto t \chi(u) dg(u) + (1-t\chi(u)) T,
\]
and hence is still an element of $\alpha$. Since $A$ coincides with $dg$ on $B_2$, it is a framing of $S_0$, which is now seen as a submanifold of $M$, and
\[
\sigma(S_0,A) = \sigma(S_0, dg) = 0.
\]
We conclude that $\tau(S_0,A)=1$.
\end{proof}

Theorem \ref{main} from the Introduction is an immediate consequence of the above result and Theorem \ref{fromold}. 

\appendix
\section{The variety of non-surjective Fredholm operators}
\label{appsec}

\setcounter{equation}{0}
\numberwithin{equation}{section}

Let $\mH$ be a real infinite dimensional Hilbert space. The symbol $\mathrm{L}(\mH)$ denotes the Banach space of bounded linear operators on $\mH$. We use the notation $\mathrm{L}(\mH_1,\mH_2)$ to denote the Banach space of bounded linear operators between different Hilbert spaces $\mH_1$ and $\mH_2$. The space $\Phi_n(\mH)$ of Fredholm operators of index $n$ on $\mH$ is an open subset of $\mathrm{L}(\mH)$.  

Let $n$ and $k$ be non-negative integers and set
\[
\Phi_n^k (\mH) := \{ F \in \Phi_n(\mH) \mid \dim \coker F = k\}.
\]
The set $\Phi_n^0(\mH)$ is open in $\Phi_n(\mH)$ and its complement, i.e. the space of non-surjective Fredholm operators of index $n$, is the disjoint union of the spaces $\Phi_n^k(\mH)$ for $k\geq 1$:
\[
\Phi_n^{\mathrm{sing}} (\mH) := \Phi_n(\mH) \setminus \Phi_n^0(\mH) = \bigsqcup_{k\geq 1} \Phi_n^k (\mH).
\]
The first aim of this appendix is to show that each $\Phi_n^k(\mH)$ is a submanifold of $\Phi_n(\mH)$ and to determine its codimension and tangent spaces.

\begin{proposition}
\label{singular}
For every non-negative integers $n$ and $k$, the set $\Phi_n^k(\mH)$ is a smooth submanifold of $\Phi_n(\mH)$ of codimension $k(n+k)$. Its tangent space at $F\in  \Phi_n^k (\mH)$ is the Banach space
\[
T_F \Phi_n^k(\mH) = \bigl\{ \widehat{F} \in \mathrm{L}(\mH) \mid (I-P) \hat{F}|_{\ker F} = 0\bigr\},
\]
where $P: \mH\rightarrow \mH$ denotes the orthogonal projector onto the image of $F$.
\end{proposition}

\begin{proof}
Fix some $F_0\in \Phi_n^k(\mH)$ and consider the orthogonal splittings
\[
\mH = \mH_0 \oplus \mH_0^{\perp} = \mH_1 \oplus \mH_1^{\perp},
\]
where
\[
\begin{split}
\mH_0 &:= \ker F_0, \qquad \dim \mH_0 = n+k, \\
\mH_1 &:= (\mathrm{im}\, F_0)^{\perp}, \qquad \dim \mH_1 = k.
\end{split}
\]
We can write $F_0$ in block matrix form as
\[
F_0 = \left( \begin{array}{cc} 0 & 0 \\ 0 & D_0 \end{array} \right): \mH_0 \oplus \mH_0^{\perp} \rightarrow \mH_1 \oplus \mH_1^{\perp},
\]
where $D_0 : \mH_0^{\perp} \rightarrow \mH_1^{\perp}$ is an isomorphism. The set
\[
U:= \left\{ F = \Bigl( \begin{array}{cc} {\scriptstyle A} & {\scriptstyle B} \\ {\scriptstyle C} & {\scriptstyle D} \end{array} \Bigr): \mH_0 \oplus \mH_0^{\perp} \rightarrow  \mH_1 \oplus \mH_1^{\perp}  \mid D: \mH_0^{\perp} \rightarrow \mH_1^{\perp} \mbox{ isomorphism} \right\}
\]
is an open neighborhood of $F_0$ in $\mathrm{L}(\mH)$. The set $U$ is contained in $\Phi_n(\mH)$ because it consists of finite rank perturbations of operators of the form
\[
\left( \begin{array}{cc} 0 & 0 \\ 0 & D \end{array} \right), \qquad \mbox{with } D : \mH_0^{\perp} \rightarrow \mH_1^{\perp} \mbox{ isomorphism,}
\]
which are Fredholm of index $n$. Let 
\[
F =  \left( \begin{array}{cc} A & B \\ C & D \end{array} \right) 
\]
be an element of $U$. Then the operator 
\[
\left( \begin{array}{cc} I_{\mH_0} & 0 \\ -D^{-1} C & I_{\mH_0^{\perp}} \end{array} \right) : \mH_0 \oplus \mH_0^{\perp} \rightarrow \mH_0 \oplus \mH_0^{\perp} 
\]
is invertible and
\[
 \left( \begin{array}{cc} A & B \\ C & D \end{array} \right) \left( \begin{array}{cc} I_{\mH_0} & 0 \\ -D^{-1} C & I_{\mH_0^{\perp}} \end{array} \right) = \left( \begin{array}{cc} A-BD^{-1} C & B \\ 0 & D \end{array} \right).
\]
The above identity shows that $F\in U$ as above belongs to $\Phi_n^k(\mH)$ if and only if it is a zero of the map
\[
\varphi : U \rightarrow \mathrm{L}(\mH_0,\mH_1), \qquad   \left( \begin{array}{cc} A & B \\ C & D \end{array} \right) \mapsto A - BD^{-1}C.
\]
The map $\varphi$ is smooth and its differential has the form
\begin{equation}
\label{ildiffe}
d\varphi(F)[\widehat{F}] = \widehat{A} - \widehat{B} DC + BD^{-1} \widehat{D} D^{-1} C - B D^{-1} \widehat{C},
\end{equation}
for every
\[
F =  \left( \begin{array}{cc} A & B \\ C & D \end{array} \right) \in U \qquad \mbox{and} \qquad \widehat{F} =  \left( \begin{array}{cc} \widehat{A} & \widehat{B} \\ \widehat{C} & \widehat{D} \end{array} \right) \in \mathrm{L}(\mH).
\]
The linear mapping $d\varphi(F): \mathrm{L}(\mH) \rightarrow  \mathrm{L}(\mH_0,\mH_1)$ is surjective for every $F\in U$, as its restriction to operators of the form
\[
 \left( \begin{array}{cc} \widehat{A} &0 \\ 0 & 0 \end{array} \right), \qquad \widehat{A} \in \mathrm{L}(\mH_0,\mH_1),
 \]
 is  clearly surjective. Since the target space $\mathrm{L}(\mH_0,\mH_1)$ is finite dimensional, the kernel of $d\varphi(F)$ is complemented in $\mathrm{L}(\mH)$ and hence $d\varphi(F)$ is a left inverse. This shows that $\varphi$ is a submersion and hence
 \[
 \Phi_n^k(\mH) \cap U = \varphi^{-1}(0)
 \]
 is a smooth submanifold of $\Phi_n(\mH)$ of codimension
\[
\codim  \Phi_n^k(\mH) \cap U = \dim \mathrm{L}(\mH_0,\mH_1) = k(n+k).
\]
The fact that $F_0\in \Phi_n^k(\mH)$ was arbitrarily chosen implies that the same is true for $\Phi_n^k(\mH)$. By (\ref{ildiffe}), the tangent space of $\Phi_n^k(\mH)$ at 
\[
F_0 = \left( \begin{array}{cc} 0 & 0 \\ 0 & D_0 \end{array} \right): \mH_0 \oplus \mH_0^{\perp} \rightarrow \mH_1 \oplus \mH_1^{\perp},
\]
is 
\[
\begin{split}
T_{F_0} \Phi_n^k(\mH) = \ker d\varphi(F_0) &= \left\{  \widehat{F} =   \Bigl( \begin{array}{cc} {\scriptsize \widehat{A}} & {\scriptsize \widehat{B}} \\ {\scriptsize \widehat{C}} & {\scriptsize \widehat{D}} \end{array} \Bigr):  \mH_0 \oplus \mH_0^{\perp} \rightarrow  \mH_1 \oplus \mH_1^{\perp} \mid \widehat{A}=0 \right\} \\ &= \bigl\{ \widehat{F} \in \mathrm{L}(\mH) \mid (I-P) \widehat{F} |_{\ker F_0} = 0 \bigr\},
\end{split}
\]
where $P : \mH \rightarrow \mH$ is the orthogonal projector onto the image of $F_0$. 
\end{proof}

We conclude this section by discussing the transversality statement that is used in the proof of Proposition \ref{tauprop}. Here, $\Phi_{n+1}(\mR \oplus \mH,\mH)$ denotes the space of Fredholm operators of index $n+1$ from $\mR \oplus \mH$ to $\mH$, and if $B$ is in $\Phi_{n+1}(\mR \oplus \mH,\mH)$ then 
\[
S_R^* B= BS_R\in \Phi_n(\mH)
\]
is the composition of $B$ with the operator 
\[
S_R: \mH \rightarrow \mR \oplus \mH,  \qquad v \mapsto (0,v).
\]

\begin{proposition}
\label{apptrans}
Let $n\geq 0$ be an integer, $\Sigma$ a finite dimensional manifold, and 
\[
G: \Sigma \rightarrow \Phi_{n+1}(\mR \oplus \mH,\mH)
\]
a smooth map such that $G(x)$ is surjective for every $x\in \Sigma$. Set
\[
F: \Sigma \rightarrow \Phi_n(\mH), \qquad F:= S_R^* G.
\]
Then for every $x\in \Sigma$ we have:
\begin{enumerate}[(i)]
\item $F(x)$ belongs to $\Phi_n^{\mathrm{sing}}(\mH)$ if and only if it belongs to $\Phi_n^1(\mH)$, if and only if $\ker G(x)$ belongs to the Grassmannian $\mathrm{Gr}_{n+1}(\mH)$ of $(n+1)$-planes in $\mR \oplus \mH$ that are contained in $(0)\oplus \mH$.
\item The map $F$ is transverse to the submanifold $\Phi_n^1(\mH)$ at $x$ if and only if the map 
\begin{equation}
\label{kermap}
\ker G: \Sigma \rightarrow \mathrm{Gr}_{n+1}(\mR\oplus \mH)
\end{equation}
is transverse at $x$ to the submanifold $\mathrm{Gr}_{n+1}(\mH)$ of $\mathrm{Gr}_{n+1}(\mR\oplus \mH)$.
\end{enumerate}
\end{proposition}

Note that the map (\ref{kermap}) is smooth because $G$ takes values in the space of surjective operators. In the proof of the above we shall make use of the following standard result in transversality theory.

\begin{lemma}
\label{transitivityoftransversality} Let $p: M \rightarrow N$ be a smooth submersion between Banach manifolds. Let $N_0$ be a finite-codimensional submanifold of $N$ and $M_0:= p^{-1}(N_0)$ the corresponding submanifold of $M$. Let $f: \Sigma \rightarrow M$ be a smooth map from a finite dimensional manifold $\Sigma$ to $M$ and set $g:= p\circ f: \Sigma \rightarrow N$. Then $f$ is transverse to $M_0$ at some $x\in \Sigma$ if and only if $g$ is transverse to $N_0$ at $x$.
\end{lemma} 

The above lemma is an immediate consequence of the corresponding statement in the linear category, whose proof is straightforward: Let $X$ and $Y$ be Banach spaces, let $P: X\rightarrow Y$ be a continuous linear operator which is a left inverse, let $Y_0$ be a finite-codimensional closed linear subspace of $Y$ and set $X_0:= P^{-1} Y_0$. Then a linear map $F: \mR^n \rightarrow X$ satisfies $\mathrm{im}\, F + X_0 = X$ if and only if the composition $PF: \mR^n \rightarrow Y$ satisfies $\mathrm{im}\, PF + Y_0 = Y$.

\begin{proof}[Proof of Proposition \ref{apptrans}.]
(i) The fact that $G(x)$ is surjective implies that the cokernel of $F(x) = S_R^* G(x)$ is at most one-dimensional. Therefore, $F(x)$ belongs to $\Phi_n^{\mathrm{sing}}(\mH)$ if and only if it belongs to $\Phi_n^1(\mH)$. Since $F(x)$ is Fredholm of index $n$, this happens if and only if the kernel of $F(x)$ has dimension $n+1$. But this is equivalent to the fact that the kernel of $G(x)$, which has always dimension $n+1$, is $(0) \oplus \ker F(x)$, i.e.\ is contained in $(0) \oplus \mH$.

(ii) The set
\[
\Phi_{n+1}^0(\mR\oplus \mH,\mH) := \{ B\in \Phi_{n+1}(\mR \oplus \mH,\mH) \mid B \mbox{ is surjective} \}
\]
is open in $\mathrm{L}(\mR \oplus \mH,\mH)$, and the map $G$ takes values into it:
\[
G: \Sigma \rightarrow \Phi_{n+1}^0(\mR\oplus \mH,\mH).
\]
The map
\begin{equation}
\label{kersub}
\ker : \Phi_{n+1}^0(\mR\oplus \mH,\mH) \rightarrow \mathrm{Gr}_{n+1}(\mR \oplus \mH), \qquad B \mapsto \ker B,
\end{equation}
is a smooth submersion, and we set
\[
\mathscr{B} := \ker^{-1} ( \mathrm{Gr}_{n+1}(\mH)) = \{ B\in \Phi_{n+1}^0(\mR\oplus \mH,\mH) \mid \ker B \subset (0) \oplus \mH\}.
\]
Note that the continuous linear mapping
\[
S_R^*: \mathrm{L}(\mR\oplus \mH,\mH) \rightarrow \mathrm{L}(\mH), \qquad B \mapsto B S_R,
\]
has the right inverse
\[
S_L^* :  \mathrm{L}(\mH) \rightarrow \mathrm{L}(\mR \oplus \mH,\mH), \qquad A \mapsto A S_L,
\]
where $S_L$ is the linear operator
\[
S_L: \mR\oplus \mH \rightarrow \mH, \qquad (s,v) \mapsto v.
\]
It follows that the restriction of $S_R^*$ to $\Phi_{n+1}^0(\mR\oplus \mH,\mH)$ onto $\Phi_n(\mH)$, which we denote by
\[
\mathscr{S} : \Phi_{n+1}^0(\mR\oplus \mH,\mH) \rightarrow \Phi_n(\mH), \qquad B \mapsto S_R^* B,
\]
is a smooth submersion. Moreover
\[
\mathscr{S}^{-1}(\Phi_n^1(\mH)) = \mathscr{B},
\] 
by the same argument used in the proof of (i). By Lemma \ref{transitivityoftransversality}, the map $F= \mathscr{S} \circ G$ is transverse to $\Phi_n^1(\mH)$ if and only if the map $G$ is transverse to $\mathscr{B}$. By a second application of Lemma \ref{transitivityoftransversality}, this time with the submersion in Equation (\ref{kersub}), the latter condition is equivalent to the fact that the map $\ker G$ is transverse to $\mathrm{Gr}_{n+1}(\mH)$.
\end{proof}


\end{document}